\documentclass{article}

\usepackage[margin=1.2in]{geometry}
\usepackage[affil-it]{authblk}

\usepackage{hyperref}
\usepackage{amsmath}
\usepackage{amsfonts}
\usepackage{amssymb}
\usepackage{mathtools}
\usepackage{amsthm}
\usepackage{eurosym}
\usepackage{xcolor}
\usepackage{graphicx}
\usepackage{tikz}
\usepackage{pgfplots}
\usepackage{dsfont}
\usepackage{pgfplots}
\usepackage{stackengine}

\usetikzlibrary{intersections, pgfplots.fillbetween}

\newcommand{\1}{\mathds{1}}

\DeclareMathOperator*{\argmax}{argmax}

\newcommand{\be}{\begin{equation}}
	\newcommand{\ee}{\end{equation}}
	\newcommand{\ov}{\overline}

\definecolor{azzurro}{rgb}{0.11, 0.67, 0.84}
\usepackage{bbm}

\theoremstyle{definition}
\newtheorem{definition}{Definition}
\theoremstyle{remark}
\newtheorem{remark}{Remark}

\newtheorem{example}{Example}
\theoremstyle{proposition}
\newtheorem{proposition}{Proposition}
\newtheorem{corollary}{Corollary}
\newtheorem{lemma}{Lemma}

\newtheorem{theorem}{Theorem}

\usepackage{url}
\usepackage{algorithm}
\usepackage{algpseudocode}
\usepackage{tikz}

\usepackage{natbib}
\bibpunct[, ]{(}{)}{,}{a}{}{,}%

\newcommand{\mc }{\mathcal }  
\newcommand{\ba}{\begin{array}}  
\newcommand{\ea}{\end{array}}  
 
\newcommand{\dd}{\,\mathrm{d}} 
\newcommand{\N}{\mathcal{N}} 

\newcommand{\R}{\mathbb{R}} 
\newcommand{\A}{\mathcal{A}}  
\usepackage{calc}
\usepackage{color}
\usepackage{dsfont}
\usepackage{graphicx}
\usepackage{epstopdf}
\usepackage{epsfig}
\usepackage{tikz}
\usepackage{enumerate,pgfplots}
\pgfplotsset{ 
	compat=newest, 
	legend style =
	{font=\footnotesize \sffamily},
	label style = {font=\small\sffamily},
	every tick label/.append style={font=\small}}
\usetikzlibrary {positioning}
\usepgfplotslibrary{fillbetween}
\usepackage{bbm} 
\usepackage{nomencl}
\makenomenclature

\newcommand{\eps}{\varepsilon}

\newcommand{\cmon}{\mathcal C^{\uparrow}_0}
\newcommand{\lip}{\hbox{Lip}_K}
\newcommand{\nexp}{\hbox{Lip}_1}

\newcommand{\aff}{\hbox{Aff}}
\newcommand{\affone}{\hbox{Aff}_1}
\newcommand{\cost}{b}
\newcommand{\ds}{\displaystyle}

\newcommand{\pa}{\tilde{p}} 
\newcommand{\ca}{\tilde{C}} 
\newcommand{\ua}{\tilde{u}}
\newcommand{\apu}{\ov u}

\newcommand{\act}{\mathcal N_\text{a}}

\title{How competitive are pay-as-bid auction games?}
\author[1]{Martina Vanelli}
\author[2\thanks{ G. Como is also with the Department of Automatic Control, Lund University, 22100 Lund, Sweden.}]{Giacomo Como}
\author[2]{Fabio Fagnani}
\affil[1]{ICTEAM institute, UCLouvain, B-1348 Louvain-la-Neuve, Belgium (email: martina.vanelli@uclouvain.be)}
\affil[2]{Department of Mathematical Sciences ``G.L.~Lagrange,'' Politecnico di Torino, 10129 Torino, Italy  (e-mail: \{giacomo.como;\,fabio.fagnani\}@polito.it).}

\date{}
\begin{document}
		\maketitle
		\begin{abstract}
	We study the pay-as-bid auction game, a supply	function model with discriminatory pricing and asymmetric firms. 
In this game, strategies are non-decreasing supply functions relating price to quantity and the exact choice
of the strategy space turns out to be a crucial issue: when it includes all non-decreasing continuous functions, pure-strategy Nash equilibria often fail to exist. 
To overcome this, we restrict the strategy space to the set of Lipschitz-continuous functions 
and we prove that Nash equilibria always exist (under standard concavity assumptions) and consist of functions that are affine on their own support and have slope equal to the maximum allowed Lipschitz constant.  We further show that the Nash equilibrium is unique up to the market-clearing price when the demand is affine and the asymmetric marginal production costs are homogeneous in zero. For quadratic production costs, we derive a closed-form expression and we compute the limit as the allowed Lipschitz constant grows to infinity. Our results show that in the limit the pay-as-bid auction game achieves perfect competition with efficient allocation 
and induces a lower market-clearing price compared to supply function models based on uniform price auctions.
			
		\end{abstract}



%

\section{Introduction.}\label{intro} 
The two primary pricing rules used in electricity spot markets are the uniform price rule and the pay-as-bid rule (\cite{wilson2002architecture,cramton2017electricity}). 
Under the uniform price rule, all transactions within a given time period are settled at a single uniform price. This means that all buyers and sellers pay or receive the same price for electricity, regardless of their individual bid or offer prices. The uniform price is typically set at the market-clearing price, which is the price at which supply equals demand and all bids and offers are matched (see the left of Figure \ref{fig:uniform-pab}). In contrast, the pay-as-bid rule settles transactions at the prices specified by the individual bids and offers submitted by participants. Each participant pays or receives the price they bid or offered, rather than a uniform price (see the right of Figure \ref{fig:uniform-pab}). 

Pricing rules can have significant implications for market outcomes and participant behavior.
In the context of electricity markets, the debate between the uniform price auction and the pay-as-bid auction emerged in the late 1990s and early 2000s, and in particular after the temporary adoption in England and Wales of discriminatory pricing in day-ahead markets (\cite{wilson2002architecture}), as well as after the ensuing debate in California (\cite{kahn2001uniform,tierney2008uniform}). 
Currently, day-ahead markets are usually structured as uniform-price auctions, while  pay-as-bid auctions are often employed in ancillary services markets and balancing markets (\cite{ahlqvist2022survey, dalkilic2013pricing, dalkilic2016dayahead}). Due to the high electricity prices experienced in recent years, the EU Agency for the Cooperation of Energy Regulators (\cite{acer2021}) considered alternative price formation models to replace uniform-price auctions in day-ahead markets, reopening the debate. 

Despite their significant economic impact, there remains a lack of understanding regarding equilibrium behavior in pay-as-bid auctions. It is worth noting that pay-as-bid auctions have received relatively less attention in the literature compared to uniform-price auctions, particularly within the context of supply function models \cite{sfe}, and most of existing results assume perfectly inelastic demand, constant/symmetric costs and/or involve solving a set of differential equations (see \cite{sfe_pab, genc, wilson2008supply, ausubel2014demand}). 
Also, understanding which type of auction is better is still an open question (\cite{zhao2023uniform, willems2022bidding, pycia2025case}).
While uniform-price auctions are criticized for causing price spikes and inefficiencies, pay-as-bid auctions may not ensure production efficiency and can favor firms with better market prediction capabilities.

\begin{figure}
	\centering
	\begin{tikzpicture}[scale=0.32]
		\draw[name path=e,scale=1,dashed,  domain=0:7.1, smooth, variable=\x] plot ({\x}, {5});
		\draw[name path=a,scale=1, domain=0:7.1, smooth, variable=\x] plot ({\x}, {0});
		
		\tikzfillbetween[of=a and e]{gray!20};	
		\draw[->] (-0.5, 0) -- (11, 0) node[right]{\small quantity};  
		\draw[->] (0, -0.5) -- (0, 11) node[above]{\small price}; 
		\draw[scale=1,red,  domain=0:10, smooth, variable=\x] plot ({\x}, {(100-\x^2)/10});
		\node[red] at (3.5,7) {\small Demand};
		\node[blue] at (3.5,3.5) {\small Supply};
		\node[red] at (3.5,6) {\small curve};
		\node[blue] at (3.5,2.5) {\small curve};
		\draw[scale=1,blue, domain=7.1:10, smooth, variable=\x] plot ({\x}, {(1/10)*(\x)^2});

		
		\draw[scale=1,blue,name path =s, domain=0:7.1, smooth, variable=\x] plot ({\x}, {(1/10)*(\x)^2});
		
		\node at (-0.5,5) {\small$ p^*$};
		\node at (7.1,-0.5) {\small $q^*$};
		\draw[scale=1,dashed, domain=0:4.9, smooth, variable=\x] plot ({7.1}, {\x});
		\node at (7.1,4.9) {\textbullet};
		
		\node at (10.9,5) {\small Market-clearing };
		\node at (10.7,4) {\small price};
	\end{tikzpicture}
	\begin{tikzpicture}[scale=0.32]
		\draw[name path=e,scale=1,dashed,  domain=0:7.1, smooth, variable=\x] plot ({\x}, {5});
		\draw[name path=a,scale=1, domain=0:7.1, smooth, variable=\x] plot ({\x}, {0});
		\draw[scale=1,blue,name path =s, domain=0:7.1, smooth, variable=\x] plot ({\x}, {(1/10)*(\x)^2});
		\tikzfillbetween[of=a and s]{gray!20};
		\draw[->] (-0.5, 0) -- (11, 0) node[right]{\small quantity};  
		\draw[->] (0, -0.5) -- (0, 11) node[above]{\small price}; 
		\draw[scale=1,red,  domain=0:10, smooth, variable=\x] plot ({\x}, {(100-\x^2)/10});
		\node[red] at (3.5,7) {\small Demand};
		\node[blue] at (3.5,3.5) {\small Supply};
		\node[red] at (3.5,6) {\small curve};
		\node[blue] at (3.5,2.5) {\small curve};
		\draw[scale=1,blue, domain=7.1:10, smooth, variable=\x] plot ({\x}, {(1/10)*(\x)^2});


		\node at (-0.5,5) {\small $p^*$};
		\node at (7.1,-0.5) {\small $q^*$};
		\draw[scale=1,dashed, domain=0:4.9, smooth, variable=\x] plot ({7.1}, {\x});
		\node at (7.1,4.9) {\textbullet};
		
		\node at (10.9,5) {\small Market-clearing};
		\node at (10.7,4) {\small price};
	\end{tikzpicture}
	\caption{On the left, the total uniform-price remuneration, on the right, the total pay-as-bid remuneration (gray-shaded regions).}
	\label{fig:uniform-pab}	
\end{figure}
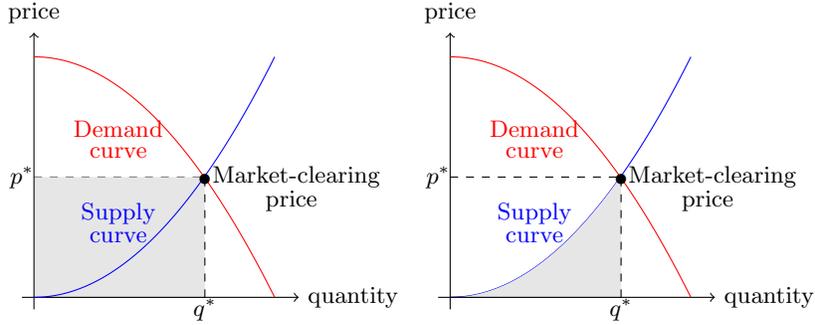

In this paper,	we  propose and analyze a supply function model with pay-as-bid remuneration and asymmetric firms, called pay-as-bid (PAB) auction game. In our model, strategies are functions relating price to quantity, in the spirit of Supply Function Equilibria game models (\cite{sfe}). In contrast to SFE, we consider pay-as-bid remuneration (as in \cite{sfe_pab} and \cite{genc}) and we do not consider uncertainty in the demand.	A crucial issue of the supply function model is that the strategy set is an infinite dimensional space. We first show that, if the strategy space includes all non-decreasing continuous functions that are zero in zero, pure strategy Nash equilibria do not exist in general. This observation stands in stark contrast to the findings of standard SFE game models, wherein it was noted that in the absence of uncertainty, the number of Nash equilibria becomes infinite. This fundamental disparity arises due to the use of pay-as-bid remuneration, where the trajectory of the supply function plays a crucial role.


Our main result is to prove that, by restricting the strategy space to Lipschitz-continuous supply functions, existence of pure-strategy Nash equilibria is guaranteed under
standard concavity assumptions and a characterization can be given. Indeed, under this assumption, the complexity of the problem can be dramatically reduced, since best responses can be characterized as piece-wise affine functions with slope equal to the maximum allowed Lipschitz constant. 
The main properties of the pay-as-bid auction game with Lipschitz-continuous 
supply functions, that we call $\lip$-pay-as-bid (PAB) auction game, can be then investigated through a subgame, the activation price  game, with a parameterized action space. 
This result  yields a fundamental simplification and paves the way to a thorough analysis of our model. 
It is essential to note that we do not assume a parametric model outright; instead, we have discovered and established the optimality of a parametric game within the domain of all Lipschitz-continuous supply functions. 
To the best of our knowledge, except the preliminary results in \cite{vanelli2023nash}, there are no similar findings available in existing literature.

Our second main contribution is on $\lip$-PAB auction games with affine demand. In this case, we prove uniqueness of Nash equilibria 
up to the market-clearing price when the asymmetric marginal costs are homogeneous in zero. 
Furthermore, for quadratic costs, 
we derive a closed-form expression for the unique Nash equilibrium 
and we compute its limit as $K$ tends towards infinity. 
This case is of utmost significance as increasing the value of $K$ widens the strategy space of Lipschitz supply functions, thereby relaxing our initial assumption. 
Indeed, the equilibrium strategies converge towards step functions that are zero up to the market-clearing price, gradually approximating the	behavior observed when studying the general case. 
In the limit, the pay-as-bid auction game achieves perfect competition and efficient allocation (\cite{dubey2009perfect}) 
and gives a lower market-clearing price than Supply Function equilibria. 

\subsection{Related literature}
{
	
	From a game-theoretic point of view, two appropriate models for studying wholesale markets for electricity are Cournot-based models and Supply Function Equilibria (SFE) game models (\cite{ventosa}). In Cournot-based models (\cite{cournot1838recherches}) firms set the quantity they want to produce. The strength of Cournot-based models is their simplicity: under standard assumption on costs and demand, Cournot equilibria exist and can be found by solving a set of algebraic equations (\cite{microeconomics}). Furthermore, they can be generalized to network frameworks (\cite{bimpikis2019cournot, cai2019role,como2024equilibria}). For these reasons, Cournot-based models are widely used in the applications. On the other hand, they have limited predictive capacity and they cannot for instance be employed to model the pay-as-bid remuneration. 
	
	Supply Function Equilibria game models represent a richer model. In these models, instead of setting the price, as in Bertrand competitions (\cite{bertrand1883review,microeconomics}), or quantities, as in Cournot models,
	firms bid a supply function relating the price to the quantity. 
	SFE models were first introduced in \cite{sfe}, and then applied to electricity markets in \cite{green}. The authors observed that, in the absence of uncertainty, there exists an infinite number of Nash equilibria. However, when the demand is uncertain and a firm faces a range of possible residual demand curves, ex-post Nash equilibria can be characterized through a unique supply function following ex-post optimal points. 
	Unlike Cournot-based models, calculating an SFE involves solving a set of differential equations. This poses limitations on the numerical tractability of SFE models and, therefore, applications mostly consider linear SFE models 
	(\cite{linear_sfe}).
	Another way to reduce the number of Nash equilibria and simplify the model is to consider SFE parametrized models (\cite{johari2011parameterized,correa2014pricing, li2015demand}). Comparisons between SFE models and Cournot-based models can be found in \cite{willems2009cournot, saglam2018ranking, delbono2016ranking}.

	While there is a vast amount of literature directed to the study of SFE in uniform-price auctions  (e.g., \cite{baldick, anderson2002optimal, anderson, holmberg2010, wilson2008supply, anderson2012asymmetric}), the behavior of SFE models is less clear when discriminatory prices are considered. 
	In \cite{sfe_pab}, a SFE game model with uncertain perfectly inelastic demand, symmetric marginal costs and pay-as-bid remuneration is proposed and analyzed. Existence of an equilibrium is ensured if the hazard rate is monotonically decreasing and sellers have non-decreasing marginal costs. They also conclude that average prices are weakly lower in
	the discriminatory auction compared to uniform-price auctions. The results are generalized in \cite{anderson2013mixed} where mixed-strategies Nash equilibria are considered and characterized. 	In \cite{genc}, SFE in discriminatory auctions are compared to uniform-price auctions when suppliers have capacity constraints. The authors formulate a supply function equilibrium model with inelastic time-varying demand and with constant symmetric marginal costs 
	and show that payments made to the suppliers 
	can be less than those in the uniform-price auction, depending on which uniform-price auction equilibrium is selected. In \cite{wilson2008supply}, SFE are characterized in a constrained transmission system for both nodal prices and pay-as-bid. 
	It is worth noting that these models assume perfectly inelastic demand and symmetric costs. To the best of our knowledge, there is a noticeable lack of models that incorporate supply functions and discriminatory pricing 
	adopting more general assumptions. Additionally, finding SFE in these models often involves solving differential equations, making them less practical for real-world applications.
	
	Alongside their application in electricity markets, pay-as-bid auctions have been studied in auction theory as alternative mechanisms to uniform-price auctions for selling treasury securities and commodities. In this context, bidders are typically assumed to be symmetric, while various forms of uncertainty and information are analyzed to evaluate their impact on bidder behavior, market efficiency, and revenue outcomes. In \cite{pycia2021auctions, pycia2025case}, the authors study pure-strategy Bayesian Nash equilibria in pay-as-bid auctions, demonstrating uniqueness and deriving a tractable representation of the equilibrium bids. Differently from uniform-price auctions, the authors establish the optimality of supply transparency and full disclosure in pay-as-bid auctions. They further show that pay-as-bid auctions are revenue dominant and potentially welfare dominant. 
	In \cite{willems2022bidding}, a perfect competition model with a continuum of generation technologies and uncertain elastic demand is developed. 
	Their findings show that pay-as-bid auctions generally lead to more competitive behavior and lower prices compared to uniform-price auctions. 
	However, when including a continuum of generation technologies, the study reveals that pay-as-bid auctions result in an inefficient generation mix. 
	A comprehensive analysis of the two auction formats is presented also in \cite{federico2003bidding}, where the authors 	analyze the two auction rules under two polar market structures (perfectly competitive and monopolistic supply), with demand uncertainty, and find that under perfect competition there is a trade-off between efficiency and the level of consumer surplus. 	In \cite{fabra, fabra2011market},	uniform auctions result in higher average prices than discriminatory auctions, but the ranking in terms of productive efficiency is ambiguous.
	
	A related problem has been studied in \cite{son2004short}, where a two-player static auction game is considered with a big player with market power and small player. The authors observe that, both under elastic and inelastic demand, the total payment of consumers would be smaller under pay-as-bid pricing for the two-player game; however, under pay-as-bid the equilibrium is a mixed-strategy equilibrium. 	In \cite{aussel2017nash} and \cite{aussel2017nash2},  Nash equilibria are fully characterized for a model of a pay-as-bid electricity market based on a multi-leader-common-follower model of a pay-as-bid
	electricity market in which the producers provide the regulator with either
	linear or quadratic bids.  
	In \cite{karaca2017game} and \cite{karaca2019designing}, the authors propose an alternative mechanism based on	Vickrey–Clarke–Groves (VCG) mechanism to obtain truthful bidding in pay-as-bid auctions.

	
	\subsection{Structure of the paper}
	The rest of the paper is structured as follows. In Section \ref{sec:model}, we introduce the pay-as-bid auction game and demonstrate that, in its general form, the game does not admit Nash equilibria. To address this, we define, in Section \ref{sec:analysis}, the $\lip$-pay-as-bid (PAB) auction game, which forms the foundation for our subsequent analysis. We then focus on the primary result: the existence and characterization of Nash equilibria for the $\lip$-PAB auction game. 	Section \ref{sec:uniqueness} is dedicated to the study of the $\lip$-PAB auction game under affine demand, showing uniqueness of Nash equilibria up to the market-clearing price when the marginal costs are homogeneous in zero. In Section \ref{sec:ld_qc}, we consider quadratic costs and we derive a closed-form expression for the unique Nash equilibrium. We then compute the limit as $K$ approaches infinity, showing that it achieves perfect competition with efficient allocation. We conclude with Section \ref{sec:conclusions}, summarizing the results and discussing future directions. 
	

\section{Model and problem formulation}\label{sec:model}
In this section, we introduce the problem setting, formally define the pay-as-bid auction game, and discuss the choice of the supply function space.

\subsection{Pay-as-bid auction games}\label{sec:PAB-games}
Consider a non-empty finite set $\N=\{1,\dots, n\}$ of producers of the same homogeneous good. 
Every producer $i$ in $\N$ is characterized by a \textit{production cost function} 
$$C_i:\R_+\to\R_+\,,$$ 
returning the cost $C_i(q_i)$ incurred when producing a quantity $q_i\ge0$ of the considered good. 
We shall assume that the production cost function $C_i$ of every producer $i$ in $\mc N$ is twice continuously differentiable, non-decreasing, and convex. 
We model the consumption market by an aggregate \textit{demand function} $$D:\R_+\to\R\,,$$ returning the quantity  $D(p)$ that consumers are willing to buy at a (maximum) unit price $p$. We assume that the aggregate demand function $D$ is twice continuously differentiable, strictly decreasing, such that 
$D(0)>0$, that the maximum price $$\hat p:=\sup\{p\ge0:\,D(p)>0\}<+\infty\,,$$
is finite, 
and that $D(p)$ is concave on the interval $[0,\hat p]$.

We assume that the producers compete strategically by bidding continuous non-decreasing \textit{supply functions} $$S_i:[0, \hat{p}]\to\R_+\,,\qquad i\in\mc N\,,$$ returning the quantities of good $S_{i}(p)$ that they are willing to sell at (minimum) unit price $p$. 
In particular, we introduce the space of continuous non-decreasing supply functions
\be\label{def-supply}
\begin{aligned}
	\cmon=\{F&
	: [0, \hat{p}]\to\R_+\,|\\&
	\, F\;\hbox{continuous non-decreasing}, F(0)=0\}\,,
\end{aligned}
\ee
and refer to the supply function $S_i$ in $\cmon$ chosen by producer $i$ in $\mc N$ as its strategy.
We stack the strategies of all the producers in a tuple of supply functions $S=(S_1,S_2,\dots,S_n)$ to be referred as a strategy profile or configuration. Following a standard notational convention in game theory, for a configuration $S$ and a producer $i$ in $\mc N$, the strategy profile of all producers but $i$ is denoted by $S_{-i} =\{S_j\}_{j \neq i}$. We shall always identify $S=(S_i, S_{-i})$.

Given a demand function $D$ and a configuration $S$, the \textit{market-clearing price} is determined as the price that makes the aggregate demand and total supply match. Specifically, the market clearing price is defined as the unique value $p^*$ in the interval $[0, \hat{p}]$  satisfying\footnote{Existence and uniqueness of the market clearing price are direct consequences of the assumptions made, specifically the fact that the demand function $D$ is continuous  strictly decreasing and $D(0)>0$, and the supply functions $S_i$ are continuous non-decreasing and $S_i(0)=0$.}
\begin{equation} \label{eq:equilibrium}
	D(p^*)= \sum_{i=1}^nS_i(p^*)\,.
\end{equation}  	
The quantity of good sold by every producer can then be computed in terms of the market-clearing price as follows
$$q^*_i = S_i(p^*)\,,\qquad i\in\mc N\,.$$ 
We shall often write $p^*=p^*(S)$ to emphasize the dependence of the market-clearing price on the configuration $S$. 
An example of market-clearing price is depicted in the left-hand side of Figure \ref{fig:paba}.
}

We introduce a class of strategic games based on the pay-as-bid remuneration, where every producer is to bid a supply function chosen from a non-empty subset $\mc A$ of the space $\cmon$:  the reason for considering restrictions in the supply function set will become clear in Section \ref{sec:Nash}.
\begin{definition}[$\mathcal{A}$-pay-as-bid auction game]\label{def:PAB} Consider a set of producers $\mathcal N$, production cost functions $C_i$ for every $i$ in $\mc N$, and  demand function $D$, as described above. For a non-empty subset of supply functions $\mathcal A \subseteq \cmon$, 
the $\mathcal A$-\textit{pay-as-bid} ($\mathcal A$-PAB) \textit{auction game} is a strategic game with player set $\N$, strategy space $\mc A$, and utility functions $u_i:\mc A^{\mc N}\to\R$ defined by
\begin{equation}\label{eq:utility_pab}
		u_i(S) :=  p^* S_i(p^*) - \int_0^{p^*}S _i(p)\dd p - C_i(S_i(p^*))
\end{equation}	
for $i$ in $\mc N$, where $p^*:=p^*(S)$ is the market-clearing price defined by \eqref{eq:equilibrium}. The strategy profile space of the $\mathcal A$-PAB auction game is denoted by $\mc S=\mc A^{\mc N}$.
\end{definition}

\begin{figure}
\centering
\begin{tikzpicture}[scale=0.22]
	\draw[->] (-0.5, 0) -- (11, 0) node[right] {$p$};
	\draw[->] (0, -0.5) -- (0, 11) node[above] {$q$}; 
	\draw[scale=1, domain=0:10, smooth, variable=\x] plot ({\x}, {10-(1/10)*\x^2});
	\draw[scale=1, domain=0:10, smooth, variable=\x, color=blue] plot ({\x}, {max(-1/4+(1/16)*\x^2,0)});
	\draw[scale=1, domain=0:8, smooth, variable=\x, color=green] plot ({\x}, {max(-1/8+(1/8)*\x^2,0)});
	\draw[scale=1, domain=0:7, smooth, variable=\x, color=red] plot ({\x}, {max(-(1/4+1/8)+ (1/16+1/8)*\x^2,0)});
	\node[red] at (6, 6.375) {\textbullet};
	\node[green] at (6, 4.375) {\textbullet};
	\node[blue] at (6, 2) {\textbullet};
	\draw[scale=1, domain=0:6.375, dashed, variable=\y, color=red] plot ({6},{\y});
	\draw[scale=1, domain=0:6, dashed, variable=\x, color=green] plot ({\x},{4.375});
	\draw[scale=1, domain=0:6, dashed, variable=\x, color=blue] plot ({\x},{2});
	\node[blue] at (-2.7,2) {$S_2(p^*)$};
	\node[green] at (-2.7,4.375) {$S_1(p^*)$};
	\node at (6,-1) {$p^*$};
	\node at (10,-1) {$\hat{p}$};
	\node at (4,10) {$D$};
	\node[blue] at (10,4) {$S_2$};
	\node[green] at (8.3,6) {$S_1$};
	\node[red] at (9.5,9){$S_1+S_2$};
\end{tikzpicture}
\begin{tikzpicture}[scale=0.22]
	\draw[name path=A,scale=1, domain=0:8, smooth, variable=\x, color=green] plot ({\x}, {max(-1/8+(1/8)*\x^2,0)});
	\draw[name path = B,scale=1, domain=0:6, dashed, variable=\x, color=green] plot ({\x},{4.375});
	\tikzfillbetween[of=A and B]{green!10};
	
	\draw[->] (-0.5, 0) -- (11, 0) node[right] {$p$};
	\draw[->] (0, -0.5) -- (0, 11) node[above] {$q$}; 
	
	\node[green] at (6, 4.375) {\textbullet};
	\draw[scale=1, domain=0:4.375, dashed, variable=\y] plot ({6},{\y});
	
	\node[green] at (-2.7,4.375) {$S_1(p^*)$};
	\node at (6,-1) {$p^*$};
	\node[green] at (8.3,6) {$S_1$};
	
\end{tikzpicture}
\caption{The market-clearing price (on the left) and the pay-as-bid remuneration (on the right).}
\label{fig:paba}
\end{figure}
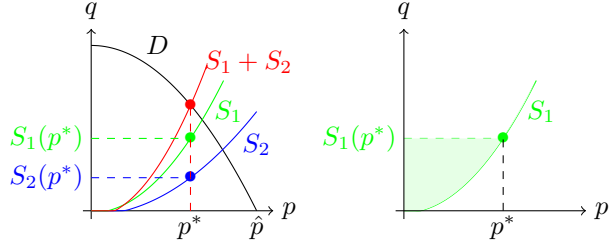  
The utility functions \eqref{eq:utility_pab} of the $\mc A$-PAB auction game can be interpreted as follows. Producer $i$ sells the total quantity $S_i(p^*)$ at the bid price and its utility is given by the difference between its total revenue and its production cost. The first two addends in the right-hand side of equation \eqref{eq:utility_pab} quantify producer $i$'s revenue with the pay-as-bid remuneration. Indeed, notice that, if the supply function $S_i$ of producer $i$ is differentiable, then
$$ p^* S_i(p^*) - \int_0^{p^*}S _i(p)\dd p=	\int_0^{p^*}p\,S'_i(p)\dd p  \,,
$$
i.e., the first two addends in  the right-hand side of \eqref{eq:utility_pab} coincide with the aggregate product of the marginal supply times the price, which is the revenue of producer $i$. 
Also, if the supply function $S_i$ is invertible, then its inverse function $P_i(q):=S^{-1}_i(q)$ is referred to as the price function and returns the marginal price $P_i(q)$ at which producer $i$ is willing to sell the quantity of good $q\ge0$, so that  its total revenue in the PAB auction game is
$$p^* S_i(p^*) - \int_0^{p^*}S _i(p)\dd p= 
\int_0^{S_i(p^*)}S_{i}^{-1}(q)\dd q 
=   \int_0^{S_i(p^*)}P_i(q)\dd q \,.
$$
Observe that formula \eqref{eq:utility_pab} for the producer's utility function in the PAB  auction game does not rely on any assumptions of differentiability or invertibility of  the supply function $S_i$. 

An example of remuneration of the PAB auction game is depicted on the right of Figure \ref{fig:paba}. 
When the supply function is $S_1(p)$ and the market-clearing price is $p^*$, the total revenue for producer $1$ coincides with the green area (the utility is then given by revenue minus the production cost). 
\begin{remark}
In contrast to the PAB auctions defined above, the total remuneration in uniform-price auctions is computed as $p^*S_i(p^*)$ (see \cite{sfe}). Therefore, the difference between the two auction formats lies in the integral term. This difference is fundamental, since in the PAB remuneration the whole trajectory of the supply function $S_i(p)$,  for $0\le p\le p^*$,  is taken into account, rather than only its value $S_i(p^*)$ at the market clearing price $p^*$. Pay-as-bid remunerations have been considered in fewer works featuring supply functions, e.g., \cite{sfe_pab, genc, pycia2020auctions, willems2022bidding}.
\end{remark}

\subsection{Best responses and Nash equilibria}\label{sec:Nash}
In the following sections, we shall focus on existence, uniqueness, and characterization of Nash equilibria of  $\mc A$-PAB auction games. A strategy profile $S$ is a  \textit{Nash equilibrium} for an $\mc A$-PAB auction game if the supply function $S_i$ of every producer $i$ in $\mc N$ maximizes its utility given the other producers' strategy profile $S_{-i}$. 
Precisely, we have the following definition. 
\begin{definition}
[Best responses and Nash equilibria]\label{defBR+Nash} 
For a non-empty subset of supply functions $\mathcal A \subseteq \cmon$, consider the $\mc A$-PAB auction game with given 
set of producers $\mathcal N$, production cost functions $C_i$ for every $i$ in $\mc N$, and  demand function $D$. 
Then: 
\begin{enumerate}
	\item[(i)]
	the \textit{best response} of a producer $i$ in $\mc N$ to a strategy profile $S_{-i}$ in $\A^{\N\setminus \{i\}}$ is the set   
	\begin{equation}\label{eq:def_br}
		\mathcal B_i(S_{-i}) = \argmax_{S_i \in \mathcal A}\  u_i(S_i, S_{-i})\,;
	\end{equation}
	\item[(ii)]
	a strategy profile $S^*$ in $\mc S$ is a (pure strategy) Nash equilibrium if 
	\be\label{eq:Nash}S^*_i\in\mathcal B_i(S^*_{-i})\,,\qquad \forall i\in\N\,.\ee
\end{enumerate}
\end{definition}

Our first observation is that, when supply functions can be generic non-increasing continuous functions, the $\cmon$-PAB auction game may not admit any Nash equilibria. 
In fact, the following result states that, in the $\cmon$-PAB  auction game, the best-response  of a producer $i$ in $ \mathcal N$ is either the constant $0$ supply function or it does not exist. Then, the only possible Nash equilibrium for the $\cmon$-PAB  auction game is the configuration where all producers bid the constant $0$ supply function and  no good is sold on the market. 
\begin{proposition}\label{pr:no_br}
Consider the $\cmon$-PAB auction game 
with set of producers $\mathcal N$, production cost functions $C_i$ for every $i$ in $\mc N$, and  demand function $D$. 
Then, for every producer $i$ in $\N$ and strategy profile $S_{-i}$, either $\mathcal B_i(S_{-i}) = \emptyset$ or $\mathcal B_i(S_{-i}) =\{0\}$.
\end{proposition} 
\begin{proof}
For a strategy profile $S$ in $\mc S$ and a producer $i$ such that $S_i\not\equiv0$, we shall prove that there exists an alternative supply function $\tilde{S}_i$ in $\cmon$ yielding  the same market-clearing price and a strictly larger utility. Formally, let $p^*=p^*(S_i, S_{-i})>0$ be the market-clearing price in configuration $S$ and define $$\tilde{S}_i(p) := S_i\left({p^2}/{p^*}\right)\,,\qquad \forall p\ge0\,.$$ Observe that $\tilde{S}_i(0)=S_i(0)$, $\tilde{S}_i(p^*)=S_i(p^*)$ and $S_i(p)\geq \tilde{S}_i(p)$ for every $p $ in $[0, p^*]$. In fact, we have that $S_i(p)= \tilde{S}_i(p)$ for every $p$ in $[0, p^*]$ if and only if $S_i\equiv0$. Therefore, if $S_i(p)\neq 0$ for some $p$ in $(0,\hat{p})$, then there exists $p_0$ in $(0, p^*)$ such that $S_i(p_0)> \tilde{S}_i(p_0)$. By continuity, this implies that there exists $\epsilon >0$ such that $S_i(p)> \tilde{S}_i(p)$ for every $p$ in $(p_0-\epsilon, p_0+\epsilon)$.  We then obtain that
$$
\int_0^{p^*}\tilde{S} _i(p)\dd p<	\int_0^{p^*}S _i(p)\dd p\,,
$$
while the other terms in the utility in \eqref{eq:utility_pab} remain constant. Consequently, we find that $u_i(\tilde{S}_i, S_{-i})> u_i(S_i, S_{-i})\,.$
Then, the best response of producer $i$ is either the constant $0$ supply function or it does not exist. 
\end{proof}
\begin{remark}\label{rem:step_f}
The main idea behind the proof of Proposition \ref{pr:no_br} is that, for every feasible supply function $S_i\not\equiv 0$ in $\cmon$, there exists another feasible supply function $\tilde{S}_i$ in $\cmon$ yielding to the same market-clearing price and a strictly higher utility. This is due to the non-compactness of the strategy space $\cmon$. The proof 
suggests that best responses would exist if one could use step functions. However, enlarging the strategy space to discontinuous functions would lead to a number of different technical difficulties as, for instance, multiple solutions of the market-clearing price equation \eqref{eq:equilibrium}. 
%
\end{remark}

\section{Existence of Nash equilibria}\label{sec:analysis}

In this section, we present our first main result, establishing the existence and providing a characterization of (pure strategy) Nash equilibria of the PAB auction game with Lipschitz-continuous supply functions. 

A supply function $F: [0, \hat{p}]\rightarrow\R_+$ is called \textit{Lipschitz-continuous} if there exists  a constant $K>0$, to be referred as its Lipschitz constant, such that 
\be\label{K-Lipschitz}
| F(x)-F(y)|\leq K|x -y|\,,\ \forall x ,y \in[0, \hat{p}] \,, \,\,x \neq y \,.\ee
Lipschitz-continuous supply functions with Lipschitz constant $K=1$ are referred to as \emph{non-expansive}. 
We denote the set of supply functions that are Lipschitz-continuous with Lipschitz constant $K>0$ by 
$\lip=\big\{F\in\cmon:\,\eqref{K-Lipschitz}\big\}$.

We shall proceed as follows. First, in Section \ref{ss:non-expansive}, we prove that existence of Nash equilibria for $\lip$-PAB games with arbitrary $K>0$ is equivalent to existence of Nash equilibria for $\nexp$-PAB games. 
Then, in Section \ref{ss:char}, we show that a best response supply function in a $\nexp$-PAB auction game is always the positive part of an affine function with unitary slope, that can be parametrized by a scalar value in $[0,\hat p]$ to be referred to as its \emph{activation price}. This leads to the introduction of an auxiliary game, the activation price game, where the players have all action space equal to the compact interval $[0,\hat p]$ and whose Nash equilibria are in bijection with those of the corresponding $\nexp$-PAB auction game. In Section \ref{ss:activgame}, we prove several properties of the activation price game, including semi-concavity of its utility functions. Finally, in Section \ref{ss:existence}, we prove the existence of Nash equilibria for the activation price game, which in turn implies that for $\lip$-PAB games.

\subsection{Reduction to non-expansive supply functions}\label{ss:non-expansive}
{We consider a $\lip$-PAB auction game (for a generic $K$) with a set of producers $\mathcal N$, production cost functions $C_i$ for every $i$ in $\mc N$, and demand function $D$. {For $\mc A=\lip$, we let $p^*:\mc A^{\mc N}\to\R_+$ and $u_i:\mc A^{\mc N} \to\R$ to be, respectively, the corresponding market clearing price and the utility functions as defined in \eqref{eq:equilibrium} and \eqref{eq:utility_pab}.} 
	We now consider the $\nexp$-PAB auction game with the same set of producers $\mathcal N$, and production cost functions and demand function specified below
	\be\label{new-game}
	\begin{aligned} C_i^{1}(q)&:=K^{-1}C_i(Kq),\; i\in\mc N,\, q\in\R^+ \\ D^1(p)&:=K^{-1}D(p)\,,\quad  p\in\R^+\,.
	\end{aligned}\ee
	We can establish a connection between the two games through the vector space isomorphism from $(\nexp)^{\mc N}$ to $(\lip)^{\mc N}$ given by the scalar multiplication
	\be\label{eq:transf}(\nexp)^{\mc N}\ni S\mapsto KS\in (\lip)^{\mc N}\ee
	We indicate with a superscript $1$, all quantities and function related to the $\nexp$-PAB auction game above, in particular, $p^{*1}$ and $u_i^1$ denote market clearing price and utility functions for this game, while $\mc B_i^1(S_{-i})$ with $S_{-i}\in (\nexp)^{\mc N\setminus\{i\}}$  indicates the usual best response set.
	\begin{proposition}\label{prop:K=1} The following facts hold for every $S\in (\nexp)^{\mc N}$.
		\begin{enumerate}
			\item[(i)] $p^{*1}(S)=p^*(KS)$ 
			\item[(ii)] $u_i^1(S)=K^{-1}u_i(KS)$ 
			\item[(iii)] $\mc B^1_i(S_{-i})=\mc B_i(KS_{-i})$
			\item[(iv)] $S$ is a Nash of the $\nexp$-PAB auction game if and only if $KS$ is a Nash of the $\lip$-PAB auction game .
		\end{enumerate}
	\end{proposition}
	
	\begin{proof}
		(i) It follows from expression \eqref{eq:equilibrium}  and the definition \eqref{new-game} for the costs and demand functions of the $\nexp$-PAB auction game that
		$$D^1(p^{*}(KS))=K^{-1}D(p^{*}(KS))=K^{-1} \sum_{i=1}^nKS^{}_i(p^{*}(KS))=\sum_{i=1}^nS^{}_i(p^{*}(KS))$$
		This implies that $p^{*}(KS)$ coincides with the market clearing price relative to the configuration $S\in (\nexp)^{\mc N}$ of the $\nexp$-PAB auction game and thus yields (i).
		
		(ii) We put $p^*=p^{*1}(S)=p^*(KS)$.
		For every producer $i$ in $\mc N$, the following derivation holds:
		$$
		\ba{rcl}\ds u_i^{1}(S^{})
		&=&\ds p^{*} S_i^{}(p^{*}) - \int_0^{p^{*}}S _i^{}(p)\dd p - C_i^{1}(S_i^{}(p^{*}))\\
		&=&\ds p^* K^{-1}KS_i(p^*) - \int_0^{p^*}K^{-1}KS _i(p)\dd p -  K^{-1}C_i(KS_i(p^*))\\
		&=&\ds K^{-1}u_i(KS)\,,\ea$$
		(iii) and (iv) are direct consequences of relation (ii).
\end{proof}}
{According to Proposition \ref{prop:K=1}, utilities, best responses and Nash equilibria of the $\nexp$-PAB auction game with demand and costs in  \eqref{new-game} are uniquely mapped to those of the $\lip$-PAB auction game  through the isomorphism \eqref{eq:transf}. Therefore, from now on, when deriving results on existence and uniqueness of Nash equilibria, we set w.l.o.g. $K=1$. 
	We apply the transformations in \eqref{new-game} and \eqref{eq:transf} when we need to characterize Nash equilibria for an arbitrary $K$, e.g., in Section \ref{sec:ld_qc} when we study the limit behavior as $K$ approaches infinity. } 

		\subsection{Characterization of best responses}\label{ss:char}	
		
		Our next result establishes that, in $\nexp$-PAB auction games, best response supply functions have an affine dependence on the price, up to the market-clearing price $p^*$. 
		
		\begin{lemma}\label{lemmma:affine}
			Consider the $\nexp$-PAB auction game with set of producers $\mathcal N$, production cost functions $C_i$ for every $i$ in $\mc N$, and  demand function $D$. 
			Then, for  every strategy profile $S$ in $\mc S=(\nexp)^{\mc N}$ and every producer $i$ in $\mc N$, there exists an activation price $x_i$ in $[0,\hat p]$ such that the non-expansive supply function $\tilde{S}_i(p)=[p-x_i]_+$ satisfies
			\be\label{same-clearing}p^*(\tilde S_i,S_{-i}^*)=p^*(S_i,S_{-i}^*)\,,\ee
			and
			\be\label{eq1}u_i(\tilde S_i,S_{-i}^*)\ge u(S_i,S_{-i}^*)\,,\ee 
			with equality if and only if 
			\begin{equation}\label{eq:br_form}
				S_i(p) =[p-x_i]_+\,,\qquad\forall p\in [0, p^*]\,.
			\end{equation}
		\end{lemma}
		\begin{proof}
			Consider a strategy profile $S$ in $\mc S$ and let $p^*=p^*(S)$ denote the associated market-clearing price. 
			For a producer $i$ in $\mc N$, consider the activation price
			$$x_i := {\sum_{j\neq i}S_j(p^*)- D(p^*)}+p^*\,,$$
			and define the supply function
			$\tilde{S}_i(p)=[p-x_i]_+$. 
			(C.f.~Figure \ref{fig:Klip_sf}.) 
			Clearly, $\tilde{S}_i\in\nexp$. 
			Notice that 
			$$\tilde S_i(p^*)=p^*-x_i=D(p^*)-\sum_{j\neq i}S_j(p^*)\,,$$
			so that \eqref{same-clearing} holds true and $\tilde S_i(p^*)=S_i(p^*)$. 
			Moreover, 
			$$S_i(p)\geq 0=\tilde{S}_i(p)\,,\qquad \forall p\in[0, x_i]\,.$$ 
			On the other hand, for every $p$ in $(x_i, p^*]$, we have that 
			$$
			\ba{rcl}
			S_i (p^*) - S_i (p) 
			&=&|S_i (p^*) - S_i (p)|
			\le|p^*-p| \\[7pt]
			&=& p^*-x_i-(p-x_i) 
			= \tilde{S}_i(p^*)-\tilde{S}_i(p)
			=S_i(p^*)-\tilde{S}_i(p)	\,,
			\ea
			$$
			where the inequality follows from the fact that $S_i\in\nexp$. Hence, we have that  
			$$S_i(p)\geq \tilde{S}_i(p)\,,\qquad \forall p\in[0, p^*]\,,\qquad  S_i(p^*)=\tilde{S}_i(p^*)\,.$$ 
			Therefore, 
			$$
			\ba{rcl}
			u_i(S_i, S_{-i}) &= & p^*S_i(p^*)- \int_0^{p^*}S_i(p)dp -C_i(S_i(p^*))\\
			&\leq & p^*\tilde{S}_i(p^*)- \int_0^{p^*}\tilde{S}_i(p)dp -C_i(\tilde{S}_i(p^*))\\&= &u_i(\tilde{S}_i, S_{-i})\,,
			\ea
			$$
			thus proving \eqref{eq1}. 
			We remark that 
			$$\int_0^{p^*}S_i(p)\dd p= \int_0^{p^*}\tilde{S}_i(p)\dd p \qquad \Leftrightarrow\qquad  S_i(p)=\tilde{S}_i(p) \,,\quad \forall p \in[0, p^*]\,,$$
			which implies that $u_i(S_i, S_{-i})=u_i(\tilde{S}_i, S_{-i})$ if and only if $S_i(p)=\tilde{S}_i(p)$ for every $p$ in $[0, p^*]$. Otherwise, $u_i(S_i, S_{-i})<u_i(\tilde{S}_i, S_{-i})$.
			This concludes the proof.
		\end{proof}
		
		\begin{proposition}[Affine best-response]\label{pr:br_form}
			Consider the $\nexp$-PAB auction game with set of producers $\mathcal N$, production cost functions $C_i$ for every $i$ in $\mc N$, and  demand function $D$. 
			For  every strategy profile $S$ in $\mc S=(\nexp)^{\mc N}$, if $S_i \in \mathcal B_i(S_{-i})$ is a best response of a producer $i$ in $\mc N$, then  there exists an activation price $x_i$ in $[0, \hat{p}]$ such that 	\eqref{eq:br_form} holds true. 
		\end{proposition}
		\begin{proof}
			The claim is a direct consequence of Lemma \ref{lemmma:affine}. 
		\end{proof}
		
		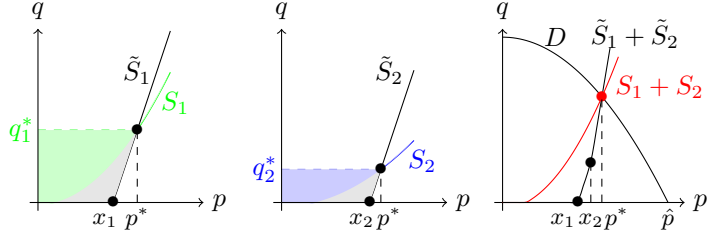
\begin{figure}
			\centering
			\begin{tikzpicture}[scale=0.22]
				\draw[name path=A,scale=1, domain=0:8, smooth, variable=\x, color=green] plot ({\x}, {max(-1/8+(1/8)*\x^2,0)});
				\draw[name path=C,scale=1, domain=0:6, smooth, variable=\x, color=black] plot ({\x}, {max(3*(\x-2/3+6.375/3-6),0)});
				\draw[name path = B,scale=1, domain=0:6, dashed, variable=\x, color=green] plot ({\x},{4.375});
				\tikzfillbetween[of=C and B]{gray!20};
				\tikzfillbetween[of=A and B]{green!20};

				\draw[->] (-0.5, 0) -- (10, 0) node[right] {$p$};
				\draw[->] (0, -0.5) -- (0, 10.5) node[above] {$q$}; 
				\draw[scale=1, domain=6:8, smooth, variable=\x, color=black] plot ({\x}, {max(3*(\x-2/3+6.375/3-6),0)});
				\node at (6, 4.375) {\textbullet};
				\node at (2/3-6.375/3+6, 0) {\textbullet};
				\draw[scale=1, domain=0:4.375, dashed, variable=\y] plot ({6},{\y});
				
				\node[green] at (-1,4.375) {$q_1^*$};
				\node at (6,-1) {\small$p^*$};
				\node at (2/3-6.375/3+6-0.5,-1.1) {\small$x_1$};
				\node[green] at (8.3,6) {$S_1$};
				\node at (6,8) {$\tilde{S}_1$};
			\end{tikzpicture}
			\begin{tikzpicture}[scale=0.22]
				\draw[name path=A,scale=1, domain=0:8, smooth, variable=\x, color=blue] plot ({\x}, {max(-1/4+(1/16)*\x^2,0)});
				\draw[name path=C,scale=1, domain=0:6, smooth, variable=\x, color=black] plot ({\x}, {max(3*(\x-4.375/3+6.375/3-6),0)});
				\draw[name path = B,scale=1, domain=0:6, dashed, variable=\x, color=blue] plot ({\x},{2});
				
				\tikzfillbetween[of=C and B]{gray!20};
				\tikzfillbetween[of=A and B]{blue!20};
				\draw[->] (-0.5, 0) -- (10, 0) node[right] {$p$};
				\draw[->] (0, -0.5) -- (0, 10.5) node[above] {$q$}; 
				
				\draw[scale=1, domain=6:8, smooth, variable=\x, color=black] plot ({\x}, {max(3*(\x-4.375/3+6.375/3-6),0)});
				\node at (6, 2) {\textbullet};
				\node at (4.375/3-6.375/3+6, 0) {\textbullet};
				\draw[scale=1, domain=0:2, dashed, variable=\y] plot ({6},{\y});
				
				\node[blue] at (-1,2) {$q_2^*$};
				\node at (6.5,-1) {\small$p^*$};
				\node at (4.375/3-6.375/3+6-0.5,-1.1) {\small$x_2$};
				\node[blue] at (8.5,2.5) {$S_2$};
				\node at (6.5,8) {$\tilde{S}_2$};
			\end{tikzpicture}
			\begin{tikzpicture}[scale=0.22]
				\draw[->] (-0.5, 0) -- (11, 0) node[right] {$p$};
				\draw[->] (0, -0.5) -- (0, 10.5) node[above] {$q$}; 
				\draw[scale=1, domain=0:10, smooth, variable=\x] plot ({\x}, {10-(1/10)*\x^2});
				\draw[scale=1, domain=0:7, smooth, variable=\x, color=red] plot ({\x}, {max(-(1/4+1/8)+ (1/16+1/8)*\x^2,0)});
				\draw[scale=1, domain=0:6.5, smooth, variable=\x] plot ({\x}, {max(3*(\x-4.375/3+6.375/3-6),0)+max(3*(\x-2/3+6.375/3-6),0)});
				\node[red] at (6, 6.375) {\textbullet};
				\draw[scale=1, domain=0:6.375, dashed, variable=\y] plot ({6},{\y});
				\draw[scale=1, domain=0:2.375, dashed, variable=\y] plot ({4.375/3-6.375/3+6},{\y});
				\node at (2/3-6.375/3+6, 0) {\textbullet};
				\node at (4.375/3-6.375/3+6, 4.375-2) {\textbullet};
				\node at (6+0.9,-1) {{\small$p^*$}};
				\node at (2/3-6.375/3+6-0.9,-1.1) {{\small$x_1$}};
				\node at (4.375/3-6.375/3+6,-1.1) {{\small$x_2$}};
				\node at (10,-1) {{\small$\hat{p}$}};
				\node at (3.2,10) {$D$};
				\node[red] at (9.5,7){$S_1+S_2$};
				\node at (8,10.2){$\tilde{S}_1+\tilde{S}_2$};
			\end{tikzpicture}
			\caption{Graphical illustration of Proposition \ref{pr:br_form}.}
			\label{fig:Klip_sf}
		\end{figure}

			A graphical interpretation of Lemma \ref{lemmma:affine} and Proposition \ref{pr:br_form} is illustrated in Figure \ref{fig:Klip_sf}. Consider two supply functions $S_1$ and $S_2$ as in Figure \ref{fig:paba}. Notice that when playing $\tilde{S}_1(p) = [p-x_1]_+$ for $x_1$ as in figure, 
			producer $1$ receives a higher utility than the one it would have obtained by playing $S_1(p)$. Indeed, its remuneration increases (colored areas), while the market-clearing price does not change, thus yielding the same quantity $\tilde S_i(p^*)=S_i(p^*)$ of good produced, hence the same total production cost $C_i(\tilde S_i(p^*))=C_i(S_i(p^*))$. This is also the case for player $2$ when playing $\tilde{S}_2(p)=[p-x_2]_+$ instead of $S_2(p)$. Then, for every strategy $S_i(p)$, it is possible to construct another supply function $\tilde{S}_i(p)$ as in \eqref{eq:br_form} yielding a higher value of the utility function. Thus, best responses must have such form up to the market-clearing price $p^*$.
			

			Proposition \ref{pr:br_form}  yields a fundamental simplification in our problem as it implies that the supply functions in every Nash equilibrium $S^*$ of a $\nexp$-PAB auction game are necessarily in the form \eqref{eq:br_form} for every producer $i$ in $\mc N$. 
			This motivates the introduction of a restricted PAB  auction game where the producers choose among supply functions in the set 
			$$\aff=\left\{F: [0, \hat p]\to\R\;|\; F(p) =[p-x]_+\;\hbox{for some}\; x\in  [0, \hat p]\right\}\,.$$
			A simple but key observation is that, differently from the space of non-expansive supply functions $\nexp$, which is infinite-dimensional, the space of affine supply functions $\aff$ is naturally embedded into a one-dimensional space. In fact, for every producer $i$, a supply function as in \eqref{eq:br_form} can be identified with its \emph{activation price} $x_i$ in $[0,\hat p]$ and every configuration $S$ in $(\aff)^{\mc N}$ can be identified with the activation price profile $x$ in $\mc X=[0,\hat p]^{\mc N}$. 
			This leads to the following definition. 
			
			\begin{definition}[Activation price game]\label{def-activation-point-game}
				Consider a set of producers $\mathcal N$, production cost functions $C_i$ for every $i$ in $\mc N$, and demand function $D$. The associated  \emph{activation price game} is the strategic game with player set $\mc N$, action space $[0,\hat p]$, and utility functions
				\begin{equation}\label{eq:u_r}
					\apu_i(x)
					= \,\ov p[\ov p-x_i]_+-\frac{1}{2}[\ov p-x_i]_+^2-C_i\left([\ov p-x_i]_+\right)
				\end{equation}
				for $i$ in $\mc N$,
				where $\ov p=\ov p(x)$ 
				is the unique market-clearing price satisfying 
				\begin{equation}\label{eq:equilibrium_r}
					D(\ov p)=\sum_{i=1}^n[\ov p-x_i]_+\,.
				\end{equation}
				The activation price profile set is denoted by $\mc X=[0,\hat p]^{\mc N}$. 
			\end{definition}

		The following result is a consequence of Lemma \ref{lemmma:affine} and Proposition \ref{pr:br_form}.
		\begin{proposition}\label{prop:restr_game} 
			Consider a set of producers $\mathcal N$, production cost functions $C_i$ for every $i$ in $\mc N$, and  demand function $D$. 
			Then:
			\begin{enumerate}
				\item[(i)] for every Nash equilibrium $x^*$ of the activation price game, the strategy profile $S^*$ with supply functions 
				\be\label{Si=p-x}S_i^*(p)=[p-x_i^*]_+\,,\qquad\forall p\in[0,\hat p]\,,\ee for every  $i$ in $\mc N$ is a Nash equilibrium of the $\nexp$-PAB auction game; 
				\item[(ii)] for every Nash equilibrium $S^*$ of the $\nexp$-PAB auction game, there exists a Nash equilibrium $x^*$ of the activation price game such that 
				\be\label{p=p}p^*(S^*)=\ov p(x^*)\,,\ee
				and 
				\be\label{Si=p-xi}S_i^*(p)=[p-x_i^*]_+\,,\qquad \forall p\in[0,p^*(S^*)]\,,\ee
				for every $i$ in $\mc N$. 
			\end{enumerate}
		\end{proposition}
		\begin{proof} (i) Let $x^*$ in $\mc X$ be a Nash equilibrium of the activation price game and let $S^*$ in $\mc S$ be a strategy profile with supply functions as in \eqref{Si=p-x} for every producer $i$ in $\mc N$. Observe that 
			\be\label{eq0}p^*(S^*)=\ov p(x^*)\,,\qquad u_i(S^*)=\ov u_i(x^*)\,.\ee
			For every supply function $S_i$ in $\nexp$, Lemma \ref{lemmma:affine} implies there exists an activation price $x_i$ in $[0,\hat p]$ such that the supply function $\tilde S_i(p)=[p-x_i]_+$ satisfies \eqref{same-clearing} and \eqref{eq1}. 
			Now, observe that 
			\be\label{eq2} 
			p^*(\tilde S_i,S_{-i}^*)=\ov p(x_i,x^*_{-i})\,,\;u_i(\tilde S_i,S_{-i}^*)=\ov u_i(x_i,x^*_{-i})\,.
			\ee
			Since $x^*$ is a Nash equilibrium of the activation price game, we have that 
			\be\label{eq3}\ov u_i(x^*)\ge\ov u_i(x_i,x^*_{-i})\,.\ee
			Equations \eqref{eq0}, \eqref{eq3}, \eqref{eq2}, and \eqref{eq1} imply that 
			$$
			u_i(S^*)=\ov u_i(x^*)\ge\ov u_i(x_i,x^*_{-i})=u_i(\tilde S_i,S_{-i}^*)\ge u(S_i,S_{-i}^*)\,.
			$$
			Since the above holds true for every producer $i$ in $\mc N$ and every supply functions $S_i$ in $\nexp$, we have that $S^*$ is a Nash equilibrium of the $\nexp$-PAB auction game. 
			
			(ii) Let $S^*$ be a Nash equilibrium of the $\nexp$-PAB auction game. Then, Proposition \ref{pr:br_form}, implies that, for every producer $i$ in $\mc N$, there exists an activation price $x_i^*$ in $[0,\hat p]$ such that \eqref{Si=p-xi} holds true. Observe that \eqref{eq0} holds true. Now, fix $x_i$ in $[0,\hat p]$ and let $S_i(p)=[p-x_i]_+$. Then, we have that 
			\be\label{eq4}p^*(S_i,S_{-i}^*)=\ov p(x_i,x^*_{-i})\,,\; u_i(S_i,S_{-i}^*)=\ov u_i(x_i,x^*_{-i})\,.\ee
			It follows from \eqref{eq0}, and \eqref{eq4} 
			that
			%
			$$\ov u_i(x^*)=u_i(S^*)\ge u_i(S_i,S^*_{-i})=\ov u_i(x_i,x^*_{-i})\,.$$
			Since the above holds true for every producer $i$ in $\mc N$ and every activation price $x_i$ in $[0,\hat p]$, we have that $x^*$ is a Nash equilibrium of the activation price game. 
		\end{proof}
		%
		%
		%

	\begin{remark}\label{remark:degenerate} In the setting of Proposition \ref{prop:restr_game}, we notice that for every Nash equilibrium $x^*$ of the activation price game, there are infinitely many Nash equilibria of the $\nexp$-PAB auction game.
		%
		%
		%
		Indeed, every configuration $\tilde S^*$ such that 
		$\tilde S^*_i(p)= [p-x_i^*]_+$ for every  $p \in [0, \ov p(x^*)]\,,$
		for every producer $i$ in $\mathcal N$, is a Nash equilibrium of the $\nexp$-PAB auction game.
	\end{remark}
	
	The following example illustrates the phenomenon discussed in Remark \ref{remark:degenerate}.
	\begin{figure}
		\centering
		\includegraphics[width=0.45\textwidth]{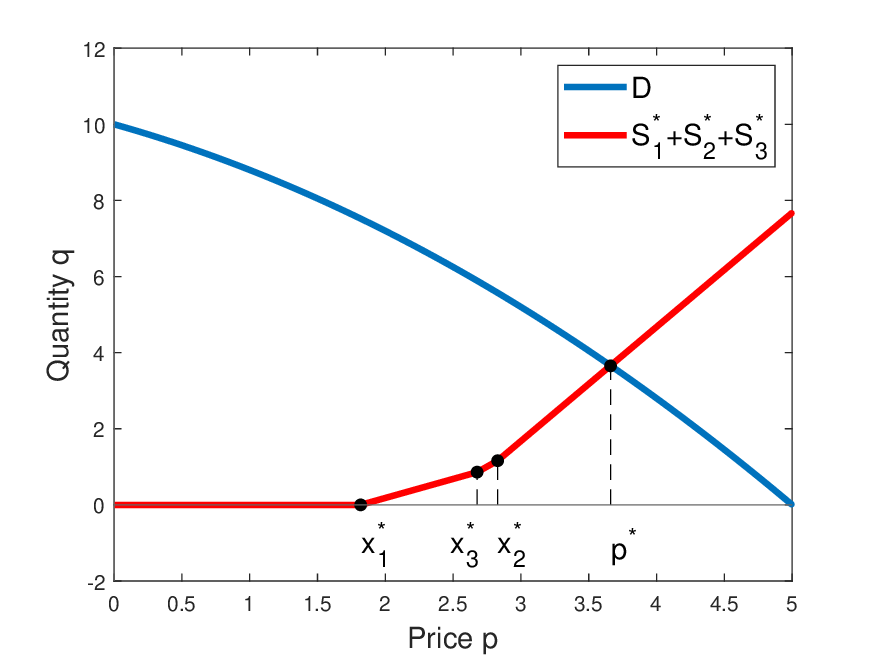}
		\includegraphics[width=0.45\textwidth]{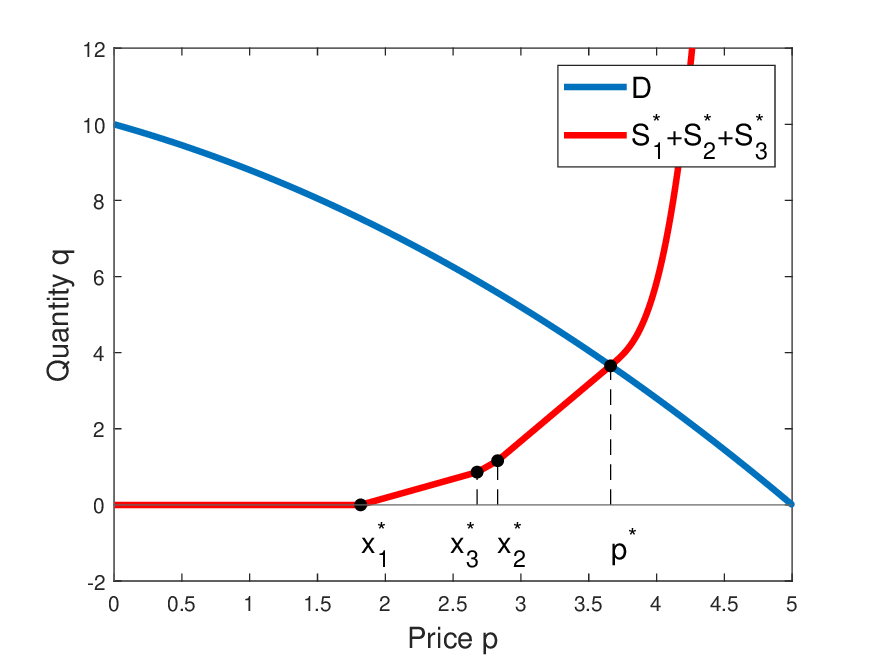}
		\caption{Two Nash equilibria of the $\nexp$-PAB auction game corresponding to the same Nash equilibrium of the activation price game (see Example \ref{ex:1_nonlin})}
		\label{fig:ex_nonlin1}
	\end{figure}
	\begin{example}\label{ex:1_nonlin}	Consider a market with demand function $$D(p)=10-p-0.2p^2$$ and three producers  with production cost functions 
		$$C_1(q)=\frac{1}{2}q+\frac{1}{4}q^2\,,\qquad C_2(q)=q+2q^2\,,\qquad C_3(q)=2q+\frac{1}{4}q^2\,.$$ 
		According to Proposition \ref{prop:restr_game}, 
		every Nash equilibrium of the $\nexp$-PAB auction game determines a vector of activation prices $x^*$ leading to a Nash equilibrium $S^{*}$ of the form $S^{*}_i(p)=[p-x_i^*]_+$, for every $i$ in $\mathcal N$. In this example, we obtain a Nash equilibrium when $x^*=[1.82,2.73, 2.65]$, thus obtaining $p^*(x^*)=3.68$ (see Figure \ref{fig:ex_nonlin1}). Notice that $x^{*}$ is a Nash equilibrium of the activation price game and corresponds to an infinite number of Nash equilibria of the $\nexp$-PAB auction game. Indeed, any non-decreasing function that is continuous and equal to $S^*$ up to $p^*$ is a Nash equilibrium of the $\nexp$-PAB auction game (see the right of Figure \ref{fig:ex_nonlin1}).
	\end{example}

	\subsection{Properties of the activation price games}\label{ss:activgame}

		Given a producer $i$ in $\mc N$ and an activation price profile $x_{-i}$ in $[0, \hat{p}]^{\mathcal N\setminus\{i\}}$ of the other producers, let
		\begin{equation}\label{hatpi} \hat p_i(x_{-i})=\ov p(\hat p, x_{-i})\,,\end{equation}
		be the market clearing price of the activation price game when producer $i$ sets its activation price equal to the maximum price $\hat p$ while the rest of the producers' activation prices profile is $x_{-i}$. As shown below, such value represents the threshold above which producer $i$ will be better off selling a $0$ quantity of good, assuming that the rest of the producers play $x_{-i}$. We start with a preliminary result that quantifies the effect of the entrance of producer $i$ in the market.
		\begin{proposition}\label{prop:increasing}
			Consider the activation price game with set of producers $\mathcal N$, production cost functions $C_i$ for every $i$ in $\mc N$,  and demand function $D$. 
			Then,  for every producer $i$ in $\mc N$: 
			\begin{enumerate}
				\item[(i)]  the function $x_i\mapsto \ov p(x_i, x_{-i})$ is non-decreasing 
				for every $x_{-i}$ in $[0,\hat p]^{\mc N\setminus\{i\}}$;  
				\item[(ii)] $\hat p_i(x_{-i})\ge\ov p(x)$;
				\item[(iii)] for every $x$ in $\mathcal{X}$, either 
				\begin{enumerate}
					\item[(a)] $\hat p_i( x_{-i})>\ov p(x)>x_i\,,$ or
					\item[(b)] $\hat p_i(x_{-i})=\ov p(x)\leq x_i\,,$
				\end{enumerate}
				holds true. 
			\end{enumerate}
		\end{proposition}
		\begin{proof}
			(i)	Let $x$ and  $\tilde x$ in $\mathcal X$ be two activation price profiles such that $x_i<\tilde x_i$ and $x_{-i}=\tilde x_{-i}$. Assume by contradiction that $\ov p(x)>\ov p(\tilde x)$.
			It follows from  \eqref{eq:equilibrium_r}  and the fact that the demand function  $D$ is strictly decreasing that
			$$ 0=D(\ov p(\tilde x))-\sum_{i=1}^n [\ov p(\tilde x)-\tilde x_i]_+> D(\ov p(x))-\sum_{i=1}^n [\ov p(x)-x_i]_+ =0\,,$$
			thus showing that $x_i\mapsto \ov p(x_i, x_{-i})$ is non-decreasing.
			
			(ii)  By definition \eqref{hatpi} and point (i), we have that 
			$$\hat p_i(x_{-i})=\ov p(\hat p,x_{-i})\ge \ov p(x_{i,x_{-i}})=\ov p(x)\,,$$
			thus proving the claim.
			
			%
				(iii) If $\hat p_i(x_{-i})=\ov p(x)$, then   \eqref{eq:equilibrium_r}, \eqref{hatpi}, and the fact that $x_i\le \hat p$ imply that 
				$$\sum_{j=1}^n[\ov p(x)-x_j]_+=D(\ov p(x))=D(\hat p_i(x_{-i}))=
				\sum_{j\ne i}[\hat p_i(x_{-i})-x_j]_+=\sum_{j\ne i}[\ov p(x)-x_j]_+\,,$$
				so that  $[\ov p(x)-x_i]_+=0$, hence $\hat p_i(x_{-i})=\ov p(x)\le x_i$. 
				
				On the other hand, if $\hat p_i(x_{-i})\ne\ov p(x)$, then, by point (ii), it must be
				\be\label{must}\hat p_i(x_{-i})> \ov p(x)\,.\ee
				Then, we have that 
				\begin{equation}\label{eq:comp}
					\ba{rcl}
					0&=&\ds D(\hat p_i(x_{-i}))-\sum_{j} [\hat p_i(x_{-i})- x_j]_+
					\\[7pt]
					&=&\ds D(\hat p_i(x_{-i}))-\sum_{j\neq i} [\hat p_i(x_{-i})- x_j]_+\\[7pt]
					&<&\ds D(\ov p(x))-\sum_{j\neq i} [\ov p(x)- x_j]_+\\[7pt]
					&=&\ds [\ov p(x)-x_i]_+	\,, 
					\ea
				\end{equation}
				where the first equality follows again from \eqref{eq:equilibrium_r} and \eqref{hatpi},
				the second one from the fact that $\hat p_i(x_{-i})=\ov p(\hat p, x_{-i})\le\hat p$
				, the central inequality follows from \eqref{must} and the fact that the demand function  $D$ is strictly decreasing, and the last equality again from \eqref{eq:equilibrium_r}. Equations \eqref{must} and \eqref{eq:comp} imply that  
				$$\hat p_i(x_{-i})>\ov p(x)>x_i\,,$$
				thus completing the proof. 
				%
			\end{proof}
			We are now ready to prove 
			quasi-concavity of producer $i$'s utility $u_i(x_i,x_{-i})$ in its activation price $x_i$ for every activation profile $x_{-i}$ of the other producers.  
			Recall that a function $f: \mc Z \rightarrow \R$ is \textit{quasi-concave} if, for every $c$ in $\R$, the upper level set $\{ x \in \mc Z \mid f(z) \geq c\}$ is convex. When, as is our case, the domain of  $f$ is an interval $\mc Z =[0, \hat{p}]\subseteq\R$, quasi-concavity can be proved by showing that $f$ is unimodal, i.e., that there exists $z^*$ in $[0, \hat{p}]$ such that $f(z)$ is monotonically non-decreasing for $z$ in $[0,z^*]$ and monotonically non-increasing for $z$ in $[z^*,\hat p]$.
			
			\begin{proposition}\label{prop:quasi-concavity}		Consider the activation price game 
				with set of producers $\mathcal N$, production cost functions $C_i$ for every $i$ in $\mc N$, and  demand function $D$. 
				Then, for every producer $i$ in $\mathcal N$ and activation price profile  $x_{-i}$ in $[0, \hat{p}]^{\mathcal N\setminus\{i\}}$ of the other producers:
				\begin{enumerate}
					\item[(i)] $x_i\mapsto \apu_i(x_i,x_{-i})$ is quasi-concave on $[0,\hat p]$ and constant on $[\hat p_i(x_{-i}),\hat p]$;
					\item[(ii)] $x_i\mapsto \apu_i(x_i,x_{-i})$ has a unique maximum point in  $[0, \hat p_i(x_{-i})]$, i.e., $$|\mathcal B_i(x_{-i})\cap [0, \hat p_i(x_{-i})]|=1\,;$$
					\item[(iii)] $\hat p_i(x_{-i})\in \mathcal B_i(x_{-i})$ if and only if $C_i'(0)\geq \hat p_i(x_{-i})$\,.
				\end{enumerate}
			\end{proposition}
			The proof of Proposition \ref{prop:quasi-concavity} is presented  in Appendix \ref{sec:proof-prop:quasi-concavity}. 
			The main technical challenges in proving Proposition \ref{prop:quasi-concavity} stem from the implicit way the utility function is defined and in the fact that utility functions are not everywhere differentiable. By making use of the Implicit Function Theorem, we show that, within any interval where the function $x_i\mapsto \apu_i(x_i,x_{-i})$ is differentiable, either its derivative $\frac{\partial}{\partial x_i}\apu_i(x_i,x_{-i})$ keeps the same sign or it undergoes a single sign change from positive to negative. This is proved by establishing that the utility is concave in the stationary points and relies on the assumptions of concavity of demand function $D(p)$ and  convexity of the production cost functions $C_i(q_i)$. Additionally, we establish that at points of non-differentiability, if there is a change
			in the derivative's sign, it necessarily transitions from positive to negative or remains positive. 
			
			{Proposition \ref{prop:quasi-concavity} has fundamental implications for existence and characterization of Nash equilibria in the activation price  game. Notice that, in Proposition \ref{prop:quasi-concavity}(iii), the best-response is characterized in terms of the threshold price $\hat p_{-i}$. In the following result, we show that this condition can be written equivalently in terms of the market-clearing price $\ov p$. } 
			\begin{proposition}\label{pr:costs_br}
				Consider the activation price game with set of producers $\mathcal N$, production cost functions $C_i$ for every $i$ in $\mc N$, and demand function $D$. 
				Then, for every activation price profile $x$ in $\mathcal X$ and producer $i$ in $\mathcal N$ such that 
				\be\label{xiinBi}x_i\in\mathcal B_i(x_{-i})\,,\ee 
				the following conditions are equivalent:
				\begin{itemize}
					\item[(a)] $x_i\geq \hat p_i(x_{-i})=\ov p(x)$;
					\item[(b)] $C_i'(0)\geq \hat p_i(x_{-i})$;
					\item[(c)] $C_i'(0)\geq \ov p(x)$.		
				\end{itemize}
			\end{proposition}
			\begin{proof}
				By Proposition \ref{prop:quasi-concavity}(i), (a) and \eqref{xiinBi} imply that $\hat p_i(x_{-i})\in\mc B_i(x_{-i})$, so that (b) follows by Proposition \ref{prop:quasi-concavity}(iii). This proves the implication (a)$\Rightarrow$(b).
				
				The implication (b)$\Rightarrow$(c) follows from Proposition \ref{prop:increasing}(ii). 
				
				To prove the implication (c)$\Rightarrow$(a), we can argue as follows. If $C_i'(0)\geq \hat p_i(x_{-i})$ then, by Proposition \ref{prop:quasi-concavity}(iii), we have that $\hat p_i(x_{-i})\in\mc B_i(x_{-i})$, so that \eqref{xiinBi} and  Proposition \ref{prop:quasi-concavity}(ii) imply that $x_i\geq \hat p_i(x_{-i})$ and (a) follows from Proposition \ref{prop:increasing}(iii). Otherwise, if 
				\begin{equation}\label{eq:case2}
					\hat p_i(x_{-i})>C_i'(0)\geq \ov p(x)\,,
				\end{equation}
				then Proposition \ref{prop:quasi-concavity} and \eqref{xiinBi} imply that 
				$ x_i<\hat p_i(x_{-i})$, so that Proposition \ref{prop:increasing} (iii) yields 
				\be\label{pixi} x_i<\ov p(x)\,,\ee
				whereas from \eqref{eq:u_r} and the fact that $C_i(z)\geq C_i'(0)z$ because of convexity of $C_i$, we obtain that
				$$
				\ba{rcl}\ds\apu_i(x_i,x_{-i})
				&=&\ds\ov p(x)[\ov p(x)-x_i]_+-\frac{1}{2}[\ov p(x)-x_i]_+^2-C_i\left([\ov p(x)-x_i]_+\right)\\[7pt]
				&\leq&\ds (\ov p(x)-C_i'(0))[\ov p(x)-x_i]_+-\frac{1}{2}(\ov p(x)-x_i)^2 \\[7pt]
				&\leq&\ds -\frac{1}{2}(\ov p(x)-x_i)^2\,.\ea
				$$
				The above and \eqref{pixi} imply that $\apu_i(x_i,x_{-i})<0=\apu_i(\hat p,x_{-i})\,,$ thus contradicting \eqref{xiinBi}.
			\end{proof}

			\subsection{Existence and characterization of Nash equilibria}
		\label{ss:existence}

		The following result establishes existence of Nash equilibria for the activation price game and characterizes 
		the set of \textit{active players}, which is defined as \begin{equation}\label{eq:act_ag}
			\act(x):=\{i\in \mathcal N\,: \,x_i < \ov p(x)\}\,.
		\end{equation}
		
		\begin{theorem}\label{th:existence}
			Consider a set of producers $\mathcal N$, production cost functions $C_i$ for every $i$ in $\mc N$, and  demand function $D$. 
			Then, 
			\begin{enumerate}
				\item[(i)] the activation price game admits a Nash equilibrium.
			\end{enumerate}
			Moreover, 
			\begin{enumerate}
				\item[(ii)]
				the set of active players 
				in any Nash equilibrium $x^*$ of the activation price game  is 
				\be\label{active}\act(x^*)=\{i \in \mathcal N\,:\, C_i'(0)<\ov p(x^*) \}\,.\ee
			\end{enumerate}
			As a consequence: 
			\begin{enumerate}
				\item[(iii)] if $C_i'(0)\leq C_j'(0)$ and
				$j\in \act(x^*)$, then $i \in \act(x^*)$;
				\item[(iv)]	if $C_i'(0)=c<\hat p$ for every $i$ in $\mathcal N$, then $\act(x^*)=\mathcal N$.
			\end{enumerate}
		\end{theorem}
		\begin{proof}
			(i) Since the action space of every player in the activation price game is the compact convex set $[0,\hat p]$, existence of Nash equilibria follows from continuity of the utility functions $\ov u_i(x)$ and their quasi-concavity in $x_i$ (Proposition \ref{prop:quasi-concavity}(i)) by a classical result in game theory, see e.g., Proposition 20.3 pp.~19--20 in \cite{osborne1994course}.  
			
			(ii) The characterization \eqref{active} follows from Proposition \ref{pr:costs_br}. 
			
			(iii) This follows directly from \eqref{active}.
			
			(iv) Suppose that $x^*$ is a Nash equilibrium and let $\ov p(x^*)$ be the corresponding market clearing price.  By contradiction, assume that $\ov p(x^*)\ge\hat p$. Then, \eqref{eq:equilibrium_r} implies that $0\ge D(\ov p(x^*))=\sum_{i=1}^n[\ov p(x^*)-x^*_i]_+$ so that $x^*_i=\ov p(x^*)=\hat p$ for every producer $i$ in $\mathcal N$. By Proposition \ref{pr:costs_br}, this would contradict the assumption $C_i'(0)=c<\hat p$ for all $i$ in $\mathcal N$. Hence, we necessarily have that $\ov p(x^*)<\hat p$. This implies that $\bar\act(x^*)$ is nonempty.  It then follows from \eqref{active} and the assumption that $C_i'(0)=c$ for every $i$ in $\mc N$ that $\act(x^*)=\mc N$. 
		\end{proof}

		Thanks to Proposition \ref{prop:restr_game} and Proposition \ref{prop:K=1}, existence of Nash equilibria of the  $\affone$-PAB  auction game guarantees existence of Nash equilibria of the $\lip$-PAB auction game. Precisely, we obtain the following conclusive result. 
		\begin{corollary}\label{coro:pab_eq}
			For every $K>0$, the $\lip$-PAB auction game with set of producers $\mathcal N$, production cost functions $C_i$ for every $i$ in $\mc N$, and  demand function $D$ 
			admits at least one Nash equilibrium $S^*$ having the form $$S^*_i(p)=K[p-x_i]_+\,,\qquad \forall p \in [0, p^*]\,,$$ for every producer $i$ in $\mathcal N$, where $p^*=p^*(S)$ is the market clearing price \eqref{eq:equilibrium} and 
			$x_i $ in $[0, \hat{p}]$ is the activation price. Furthermore, for every Nash equilibrium $S^*$,  
			$$
			S^*_i(p^*) >0 \quad \Leftrightarrow \quad C_i'(0)<p^*\,.
			$$
		\end{corollary}
		\begin{proof} Direct consequence of Theorem \ref{th:existence}, Proposition \ref{prop:restr_game} and Proposition \ref{prop:K=1}. 
	\end{proof}
	
	
	\section{Uniqueness of Nash equilibria}\label{sec:uniqueness}
	In this section, we focus on the special case of affine demand function, i.e.,
	\begin{equation}\label{eq:aff_dem}
		D(p)=\gamma (\hat p-p)\,,
	\end{equation}
	for some $\gamma >0$ and $\hat p>0$, and investigate uniqueness of Nash equilibria for the activation price game.		
	We show that, when the producers' marginal costs in zero are homogeneous, i.e., if 
	\be\label{eq:homogeneous-marginal-copsts}C_i'(0)=C_j'(0)\,,\qquad \forall i,j\in\mc N\,,\ee 
	then the Nash equilibrium $x^*$ of the activation price game is unique. 
	By Propositions \ref{prop:restr_game} and \ref{prop:K=1}, this implies that for every $K>0$, the Nash equilibrium $S^*$ of the $\lip$-PAB auction game is unique up to the clearing price $p^*$. 
	We also show examples of non-uniqueness of Nash equilibria arising when the producers marginal costs in zero are heterogeneous. 

\subsection{All-active games}
Theorem \ref{th:existence} motivates the interest in studying Nash equilibria $x^*$ of the activation price game where all producers are active, i.e., 
where $\act(x^*)=\mc N$. In this subsection, we introduce a related family of super-modular games (c.f.~\cite{Topkins:1979,Milgrom.ROberts:1990,Vives:1990,Topkins:1998, amir2005super}) that admit a unique Nash equilibrium. Such games turn out to be instrumental in studying uniqueness and properties of Nash equilibria of the original $\lip$-PAB auction game. 

\begin{definition}[All-active game]\label{def:all-active} Consider a set of producers $\mathcal N$, production cost functions $C_i$ for every $i$ in $\mc N$, and affine demand function \eqref{eq:aff_dem}. 
The \textit{all-active game} is a strategic game with player set $\N$, strategy space $[0,\hat p]$, and utility functions
\begin{equation}\label{eq:ut_aa}
	\begin{array}{rcl}\ua_i(x)&=&\ds\pa(x)(\pa(x)-x_i)-\frac12{(\pa(x)-x_i)^2}\\[10pt]&&\ds-\ca_i(\pa(x)-x_i)\\[10pt]
		&=&\ds\frac12\left(\pa(x)^2-x_i^2\right) -\ca_i(\pa(x)-x_i)\,,
	\end{array}
\end{equation}
for every producer $i$ in $\N$, where 
\begin{equation}\label{eq:c_ext}
	\ca_i(z)= \begin{cases}
		C_i(z)\quad &\text{if }z\geq 0\\
		C_i(0)+C_i'(0)z+\frac{1}{2}C_i''(0)z^2 \quad &\text{if }z<0 \,,
	\end{cases}
\end{equation}
is an extended production cost function and 
\begin{equation}\label{eq:p_tilde}
	\pa(x)=\frac{\gamma\hat p+\sum_{j=1}^n x_j}{n+\gamma}\,,
\end{equation}	
is the corresponding market-clearing price. 
\end{definition}

{We recall that a game with a compact action space $\mc A\subset \R$ is \textit{super-modular} when  its utilities are upper semi-continuous and satisfy the \textit{increasing difference property}, that is, for all $i$ in $\mc N$, $x_i'\geq x_i$ and $x_{-i}'\geq x_{-i}$, 
\begin{equation}\label{eq:incr_diff}
	u_i(x_i',x_{-i}')-u_i(x_i,x_{-i}')\geq u_i(x_i',x_{-i})-u_i(x_i,x_{-i})\,.
\end{equation} When $\mc A$ is an interval and the utilities are twice differentiable, this condition is satisfied if and only if $\frac{\partial^2 u_i}{\partial x_i \partial x_j}(x)\geq 0$, for every  $x$  in $\mc X$, and $i$ and $j$ in $\mc N$, $i\neq j$.} 
For all-active games, we have the following fundamental fact.
\begin{proposition}\label{pr:all_sup}
Consider a set of producers $\mathcal N$, production cost functions $C_i$ for every $i$ in $\mc N$, and affine demand function \eqref{eq:aff_dem}. 
Then:
\begin{enumerate}
	\item[(i)]   for every producer $i$ in $\mc N$, the utility function $\ua_i(x)=\ua(x_i,x_{-i})$, defined in \eqref{eq:ut_aa}, is strictly concave in $x_i$ for every action profile $x_{-i}$ in $[0,\hat p]^{\mc N\setminus\{i\}}$;
	\item[(ii)] the all-active game is super-modular; 
	\item[(iii)] the all-active game admits a unique Nash equilibrium $x^*$ in $[0,\hat p]^{\mc N}$. 
\end{enumerate}
\end{proposition}
\begin{proof}
For every producer $i$ in $\mc N$, equations \eqref{eq:ut_aa}, \eqref{eq:c_ext}, and \eqref{eq:p_tilde} imply that 
$$
\ba{rcl}
\ds\frac{\partial \ua_i(x)}{\partial x_i }
&=&\ds\pa(x)\frac{\partial \pa(x)}{\partial x_i }-x_i -\ca_i'(\pa(x)-x_i)\frac{\partial \pa(x)}{\partial x_i } \\[10pt]
&=&\ds\frac{\pa(x)}{n+\gamma}-y_i +\left(\frac{n+\gamma-1}{n+\gamma}\right)C_i'(\pa(x)-x_i)\,,\ea
$$
and
$$
	\ba{rcl}
	\ds \frac{\partial^2 \ua_i(x)}{\partial x_i^2}
	&=&\ds\ds\frac{1}{n+\gamma}\frac{\partial \pa(x)}{\partial x_i }-1 +\left(\frac{n+\gamma-1}{n+\gamma}\right)C_i''(\pa(x)-x_i)\frac{\partial \pa(x)}{\partial x_i }\\[10pt]
	&=&\ds\frac{1}{(n+\gamma)^2}-1 -\left(\frac{n+\gamma-1}{n+\gamma}\right)^2C_i''(\pa(x)-x_i)<0
	\,,\ea
$$
where the last inequality follows from the fact that $n\ge1$, $\gamma>0$, and $C_i''(q_i)\ge0$ since the cost function $C_i$ is convex. The fact that ${\partial^2 \ua_i}/{\partial x_i ^2}<0$ readily implies point (i) of the claim. 

On the other hand, for every $i\ne j$ in $\mc N$, we have 
\be\label{extra-diag}\ba{rcl}
\ds \frac{\partial^2 \ua_i(x)}{\partial x_i \partial x_j}
&=&\ds\ds\frac{1}{n+\gamma}\frac{\partial \pa(x)}{\partial x_j} \\[10pt] &&\ds+\left(\frac{n+\gamma-1}{n+\gamma}\right)C_i''(\pa(x)-x_i)\frac{\partial \pa(x)}{\partial x_j }\\[10pt]
&=&\ds\frac{1}{(n+\gamma)^2} +\frac{(n+\gamma-1)}{(n+\gamma)^2}C_i''(\pa(x)-x_i)
\\[10pt]
&>&0
\,,\ea\ee
where the last inequality again follows from the fact that $n\ge1$, $\gamma>0$, and $C_i''(q_i)\ge0$. 
Equation \eqref{extra-diag} implies that the game is super-modular, thus proving point (ii) of the claim. 

Finally, we have that 
\be\label{diag-dom}\ba{rcl}
\small	\ds\sum_{j=1}^n\frac{\partial^2 \ua_i(x)}{\partial x_i \partial x_j}
&=&\ds\frac{n}{(n+\gamma)^2}-1\\[10pt] &&\ds-\frac{\gamma(n+\gamma-1)}{(n+\gamma)^2}C_i''(\pa(x)-x_i)<0
\,,\ea\ee
where the last inequality follows again from the fact that $n\ge1$, $\gamma>0$, and $C_i''(q_i)\ge0$.
Notice that \eqref{extra-diag} and \eqref{diag-dom} imply that 
$$\ds\frac{\partial^2 \ua_i(x)}{\partial x_i^2}
+\sum_{j\ne i}\left|\frac{\partial^2 \ua_i(x)}{\partial x_i \partial x_j}\right|<0\,.$$
By \cite[Theorem 4.1]{Gabay.Moulin:80}, the strict diagonal dominance above implies uniqueness of the Nash equilibrium, thus proving point (iii) of the claim. 
\end{proof}

Observe that, with respect to the activation price game (Definition \ref{def-activation-point-game}), in defining the all-active game we drop the positive part of the difference $\pa(x)-x_i$ between the market clearing price and the activation price both in the utility function \eqref{eq:ut_aa} (c.f.~\eqref{eq:u_r}) and in the corresponding the market clearing price \eqref{eq:p_tilde} (c.f.~\eqref{eq:equilibrium_r}). The following connection between the activation price game and the all-active game can be established.
\begin{proposition}\label{prop:game-comp1} 
For a set of producers $\mathcal N$, production cost functions $C_i$ for every $i$ in $\mc N$, and affine demand function \eqref{eq:aff_dem}, consider the activation price game with clearing price $\ov p(x)$ and utility functions $\ov u_i(x)$, and the all-active game with clearing price $\tilde p(x)$ and utility functions $\tilde u_i(x)$. 
Then, for every activation price profile $x$ in $\mathcal X$, the following conditions are equivalent:
\begin{enumerate}
	\item[(a)] $x_i\leq \ov p(x)$ for every $i$ in $\mc N$;
	\item[(b)] $x_i\leq \pa(x)$ for every $i$ in $\mc N$.
\end{enumerate}
Moreover, if any of the above conditions is satisfied, then \be\label{p*=pa}\ov p(x)=\pa(x)\,,\ee and \be\label{u=u*}\apu_i(x)=\ua_i(x)\,,\ee
for every producer $i$ in $\mc N$. 
\end{proposition}
\begin{proof}
If condition (a) is satisfied, then \eqref{eq:equilibrium_r} and \eqref{eq:aff_dem} imply that
$$\sum_{i=1}^n(\ov p(x)-x_i)=\sum_{i=1}^n[\ov p(x)-x_i]_+=D(\ov p(x))=\gamma (\hat p-\ov p(x))\,,$$
so that condition (a) and \eqref{eq:p_tilde} yield
$$x_i\le \ov p(x)=\frac{\gamma\hat p+\sum_{j=1}^n x_j}{n+\gamma}=\pa(x)\,,\;\forall i\in\mc N$$  
This proves that condition (a) implies that both \eqref{p*=pa} and condition (b) are satisfied.  

Conversely, if condition (b) is satisfied, then \eqref{eq:equilibrium_r} and \eqref{eq:aff_dem} imply that
$$\sum_{i=1}^n[\pa(x)-x_i]_+=\sum_{i=1}^n(\pa(x)-x_i)=\gamma(\hat p-\pa(x))=D(\pa(x))\,,$$
so that $\pa(x)$ satisfies equation \eqref{eq:equilibrium_r} for the market clearing price of the activation price game. Since equation \eqref{eq:equilibrium_r} admits $\ov p(x)$ as its unique solution, this implies that \eqref{p*=pa} is satisfied, 
so that condition (b) and \eqref{eq:p_tilde} yield
$x_i\le\pa(x)=\ov p(x)\,,$
for every producer $i$ in $\mc N$, thus proving that condition (a) is satisfied.  

Finally, observe that, for every $x$ such that conditions (a) and (b) are verified, we have that  $\pa(x)=\ov p(x)$, so that \eqref{eq:u_r} and \eqref{eq:ut_aa} imply that 
$$\ba{rcl}\apu_i(x)
&=&\ov p[\ov p(x)-x_i]_+-\frac{1}{2}[\ov p(x)-x_i]_+^2-C_i\left([\ov p(x)-x_i]_+\right)\\[7pt]
&=&\pa(x)(\pa(x)-x_i)-\frac12{(\pa(x)-x_i)^2}-\ca_i(\pa(x)-x_i)\\[7pt]
&=&\ua_i(x)\,,
\ea$$
for every producer $i$ in $\mc N$, thus proving \eqref{u=u*}. 
\end{proof}

We now want to link our result on all-active games to the activation price games. The following result provides a sufficient condition for the Nash equilibria of the all-active game and the activation price game to coincide.

\begin{proposition}\label{prop:game-comp2} 
Consider a set of producers $\mathcal N$, production cost functions $C_i$ for every $i$ in $\mc N$, and affine demand function \eqref{eq:aff_dem}. 
Then, every Nash equilibrium $x^*$ of the activation price game such that $\act(x^*)=\mc N$
is also a Nash equilibrium of the all-active game.
\end{proposition}
\begin{proof}
Since the clearing price $\ov p(x)$ of the activation price game is a continuous function of the activation price profile $x$, we have that for activation profiles $x$ sufficiently close to the Nash equilibrium $x^*$, condition (a) of Proposition \ref{prop:game-comp1} is satisfied, so that \eqref{u=u*} holds true. In particular, for every producer $i$ in $\mc N$,  there exists $\eps>0$ such that 
$$ \ua_i(x_i, x^*_{-i})= u(x_i, x^*_{-i})\leq u_i(x^*_i, x^*_{-i})=\ua_i(x^*_i, x^*_{-i})\,,$$
for every $x_i$ in $(x_i^*-\epsilon, x_i^*+\epsilon)$. 
This implies that $x^*_i$ is a local maximum of  the function $x_i\mapsto \ua_i(x_i, x^*_{-i})$. By Proposition \ref{pr:all_sup}, the function $x_i\mapsto \ua_i(x_i, x^*_{-i})$ is strictly concave, so that $x^*_i$ is indeed the global maximum point.  This proves that, in the all active game, $x^*_i$ is a best response to $x_{-i}^*$ for every producer $i$ in $\mc N$,  i.e., $x^*$ is a Nash equilibrium of the all-active game.
\end{proof}

We can now state the main result of this section.
\begin{theorem}\label{th:eq_cost} 
Consider a set of producers $\mathcal N$, production cost functions $C_i$ for every $i$ in $\mc N$, and affine demand function \eqref{eq:aff_dem}. 
If the marginal costs in $0$ are homogeneous, i.e., if \eqref{eq:homogeneous-marginal-copsts} holds true, then the activation game admits a unique Nash equilibrium $x^*$. Moreover, the activation prices at equilibrium satisfy
\begin{equation}\label{eq:part_du}
	\pa(x^*)-(n+\gamma)x^*_i +\left(n+\gamma-1\right)C_i'(\pa(x^*)-x^*_i)=0
\end{equation}
for $i$ in $\mc N$, where $\tilde p$ is given by \eqref{eq:p_tilde}.

\end{theorem}
\begin{proof}
By Theorem \ref{th:existence}(iv), every Nash equilibria $x^*$ of the activation price game has active set $\act(x^*)=\mathcal N$, so that Proposition \ref{prop:game-comp2} implies that it must also be a Nash equilibrium for the all-active game. That the activation price game admits a unique Nash equilibrium then follows from uniqueness of Nash equilibria of the all-active game (Proposition \ref{pr:all_sup}). Equation \eqref{eq:part_du} then follows from the first order conditions $\frac{\partial \tilde u_i}{\partial x_i}(x^*)=0$ on the utility functions \eqref{eq:ut_aa} of the all-active game.
\end{proof}

Theorem \ref{th:eq_cost} establishes uniqueness of Nash equilibria for the activation price game when the marginal costs are homogeneous in zero. This assumption guarantees that at the Nash equilibrium the players are either all active or all inactive. Thanks to Proposition \ref{prop:K=1} and Proposition \ref{prop:restr_game}(ii), Theorem \ref{th:eq_cost} implies the following result on the uniqueness of the Nash equilibrium of the $\lip$-PAB auction game up to the clearing price. 

\begin{corollary}\label{coro:uniqueness}
Consider a set of producers $\mathcal N$, production cost functions $C_i$ for every $i$ in $\mc N$, and affine demand function \eqref{eq:aff_dem}. 
If the marginal costs in $0$ are homogeneous, i.e., if \eqref{eq:homogeneous-marginal-copsts} holds true, 
then, for every $K>0$, there exists a unique $x^*$
such that every Nash equilibrium $S^*$ of the $\lip$-PAB auction game satisfies, for $i$ in $\mc N$, \be\label{unique-eq}S^*_i(p)=K[p-x^*_i]_+\,,\qquad \forall p \in [0, p^*(S^*)]\ee
\end{corollary}
\begin{proof}
	The claim follows from Propositions \ref{prop:K=1} and \ref{prop:restr_game}(ii), and Theorem \ref{th:eq_cost}. 
\end{proof}

In the following, we will present some examples with some intrinsic multiplicity phenomena that can appear when this condition is not met.
In particular, the following example shows that, if there are agents that sell a zero quantity, then there can exist multiple Nash equilibria of the activation price game (hence, for the $\lip$-PAB auction game), though equivalent from the point of view of utilities. 

\begin{figure}
	\centering
	\includegraphics[width=0.4\textwidth]{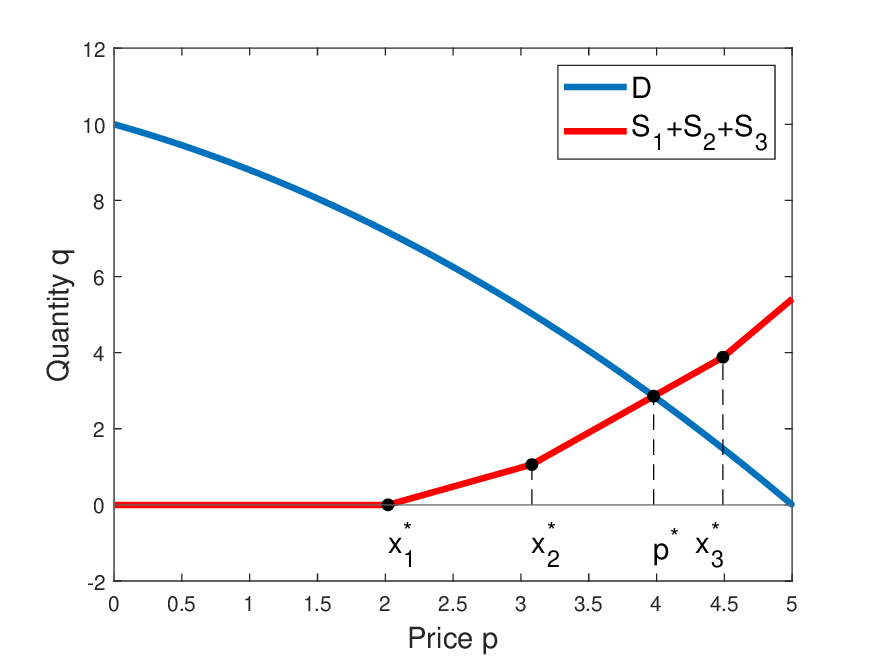}
	\includegraphics[width=0.4\textwidth]{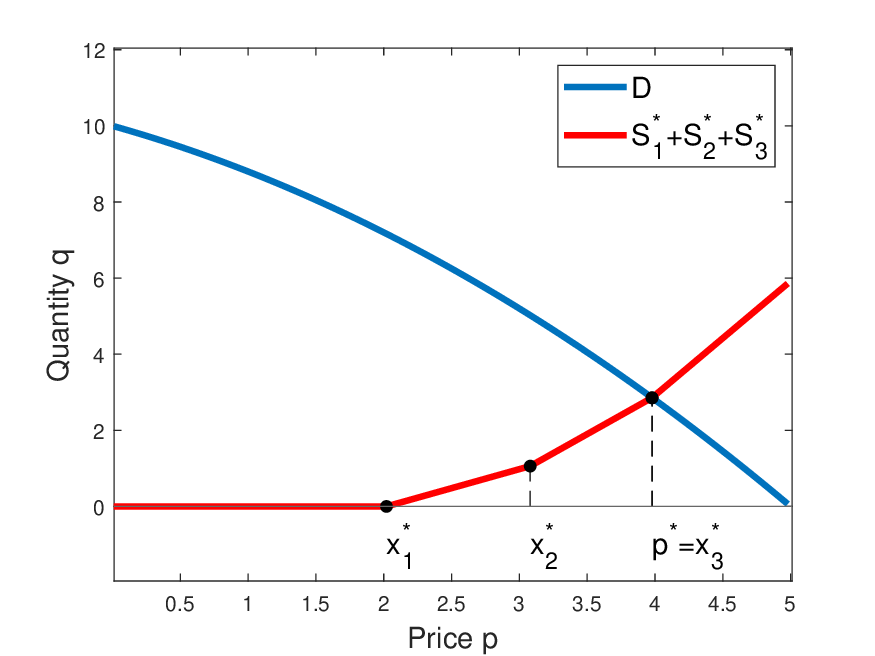}
	\caption{Two Nash equilibria of both the activation price  game and the $\lip$-PAB auction game (see Example \ref{ex:2_nonlin})}
	\label{fig:ex_nonlin2}
\end{figure}
\begin{example}\label{ex:2_nonlin}
	Let us consider the same setting as in Example \ref{ex:1_nonlin}, except for the production cost of producer $i=3$ that is now $C_3(q)=4q+\frac{1}{4}q^2$. In this case, we obtain that the vector of activation prices $x^*=[2.08,3.17, 4.5]$ corresponds to a Nash equilibrium of the activation price game with market-clearing price $p^*(x^*)=4$ (see Figure \ref{fig:ex_nonlin2}). Since $x^*_3>p^*(x^*)$, we have that, at this Nash equilibrium, agent $3$ sells a zero quantity and has no influence on the market-clearing price. 
	Notice that all vectors $y=[x_1^*, x_2^*, y_3]$ where $p^*(x^*)\leq y_3 \leq x^*_3$ are Nash equilibria. 
\end{example}
On the other hand, we could wonder whether multiple Nash equilibria that are not equivalent from the point of view of utilities do exist. In the following example, we shall observe that even in very simple scenarios, there can be a multiplicity of Nash equilibria that are not equivalent from the point of view of utilities. 

\begin{example}
	Let $D(p)=1-p$ and $\mathcal N=\{1,2\}$ with $C_1(q)=\frac{1}{2}q^2$ and $C_2(q)=c_2q+\frac{1}{2}q^2$. 
	Then, the utility of firm $1$ in the activation price game is given by
	$$
	\begin{aligned}
		\bar u_1(x_1, x_2)
		&=[p^*(x)-x_1]_+x_1\\&=
		\begin{cases}
			\frac{1}{2}\left(1-x_1\right)x_1 \,\, &\text{if }0\leq x_1 \leq 2x_2-1
			\\
			\frac{1}{3}(1+x_2-2x_1)x_1 \,\,&\text{if }2x_2-1<x_1 \leq (x_2+1)/2\\
			0 &\text{if }(x_2+1)/2< x_1 \leq 1\,,
		\end{cases}
	\end{aligned}
	$$
	which leads to the best response
	\begin{equation}\label{eq:br1_ex}
		\mathcal B_1(x_2)=\begin{cases}
			(1+x_2)/4 \quad &\text{if }x_2< \frac{5}{7}\\
			2x_2-1 \quad &\text{if }\frac{5}{7}\leq x_2 < \frac{3}{4}\\
			\frac{1}{2} &\text{if }x_2 \geq \frac{3}{4}
		\end{cases}
	\end{equation}
	Similarly, 
	we can derive that the best response of firm $2$ is given by
	\begin{equation}\label{eq:br2_ex}
		\mathcal B_2(x_1)=\begin{cases}
			\left[\frac{x_1+1}{2},1\right] &\text{if }x_1< 2c_2-1\\
			\frac{1+x_1+2c_2}{4} \quad &\text{if }2c_2-1\leq x_2< \frac{5+2c_2}{7}\\
			2x_1-1 \quad &\text{if }\frac{5+2c_2}{7}\leq x_2 < \frac{3+c_2}{4}\\
			\frac{1+c_2}{2}\quad &\text{if }x_1 \geq \frac{3+c_2}{4}
		\end{cases}
	\end{equation}
	
	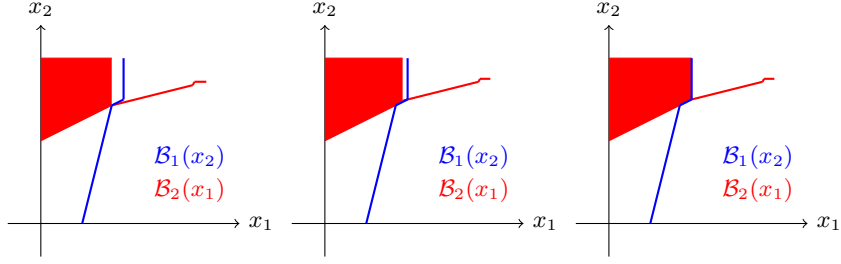
\begin{figure}
		\centering
		\begin{tikzpicture}[scale=2.2]
			\draw[->] (-0.2, 0) -- (1.2, 0) node[right] {\small $x_1$};
			\draw[->] (0, -0.2) -- (0, 1.2) node[above] {\small $x_2$}; 
			
			\draw[name path=A,domain=0:3/7, variable=\x, color=red] plot ({\x}, {(\x+1)/2});
			\draw[name path=B,domain=0:3/7,variable=\x, color=red] plot ({\x}, {1});
			\tikzfillbetween[of=A and B]{red}
		\draw[domain=3/7:45/49,thick, variable=\x, red] plot ({\x},{(17+7*\x)/28});
		\draw[domain=45/49:13/14,thick, variable=\x,red] plot ({\x},{2*\x-1});
		\draw[domain=13/14:1, thick,variable=\x,red] plot ({\x},{6/7});
		
		\draw[domain=0:5/7,thick, variable=\x, blue] plot ({(1+\x)/4}, {\x});
		\draw[domain=5/7:3/4,thick, variable=\x,blue] plot ({2*\x-1}, {\x});
		\draw[domain=3/4:1, thick,variable=\x,blue] plot ({1/2}, {\x});
		
		
		\node[blue] at (0.9,0.4){\small $\mathcal B_1(x_2)$};
		\node[red] at (0.9,0.2){\small $\mathcal B_2(x_1)$};
	\end{tikzpicture}
	\begin{tikzpicture}[scale=2.2]
		\draw[->] (-0.2, 0) -- (1.2, 0) node[right] {\small $x_1$};
		\draw[->] (0, -0.2) -- (0, 1.2) node[above] {\small $x_2$}; 
		
		\draw[name path=A,domain=0:1/2-0.03, variable=\x, color=red] plot ({\x}, {(\x+1)/2});
		\draw[name path=B,domain=0:1/2-0.03,variable=\x, color=red] plot ({\x}, {1});
		\tikzfillbetween[of=A and B]{red}
	\draw[domain=1/2-0.03:13/14,thick, variable=\x, red] plot ({\x},{(5+2*\x)/8});
	\draw[domain=13/14:15/16,thick, variable=\x,red] plot ({\x},{2*\x-1});
	\draw[domain=15/16:1, thick,variable=\x,red] plot ({\x},{7/8});
	
	\draw[domain=0:5/7,thick, variable=\x, blue] plot ({(1+\x)/4}, {\x});
	\draw[domain=5/7:3/4,thick, variable=\x,blue] plot ({2*\x-1}, {\x});
	\draw[domain=3/4:1, thick,variable=\x,blue] plot ({1/2}, {\x});
	
	\node[blue] at (0.9,0.4){\small $\mathcal B_1(x_2)$};
	\node[red] at (0.9,0.2){\small $\mathcal B_2(x_1)$};
	
\end{tikzpicture}
\begin{tikzpicture}[scale=2.2]
	\draw[->] (-0.2, 0) -- (1.2, 0) node[right] {\small $x_1$};
	\draw[->] (0, -0.2) -- (0, 1.2) node[above] {\small $x_2$}; 
	\draw[name path=A,domain=0:1/2, variable=\x, color=red] plot ({\x}, {(\x+1)/2});
	\draw[name path=B,domain=0:1/2,variable=\x, color=red] plot ({\x}, {1});
	\draw[domain=1/2:13/14,thick, variable=\x, red] plot ({\x},{(5+2*\x)/8});
	\draw[domain=13/14:15/16,thick, variable=\x,red] plot ({\x},{2*\x-1});
	\draw[domain=15/16:1, thick,variable=\x,red] plot ({\x},{7/8});
	\tikzfillbetween[of=A and B]{red};
	\draw[domain=0:5/7,thick, variable=\x, blue] plot ({(1+\x)/4}, {\x});
	\draw[domain=5/7:3/4,thick, variable=\x,blue] plot ({2*\x-1}, {\x});
	\draw[domain=3/4:1, thick,variable=\x,blue] plot ({1/2}, {\x});
	\node[blue] at (0.9,0.4){\small $\mathcal B_1(x_2)$};
	\node[red] at (0.9,0.2){\small $\mathcal B_2(x_1)$};
\end{tikzpicture}
\caption{Best responses and Nash equilibria for $c_2=\frac{5}{7}$, $c_2\in(\frac{5}{7},\frac{3}{4})$ and $c_2=\frac{3}{4}$, respectively.}
\label{fig:br}
\end{figure}

We then consider the following three cases.
\begin{itemize}
\item If $c_2\in [0, \frac{5}{7}]$, we obtain the unique Nash equilibrium
\begin{equation}\label{eq:eq_2ac}
	x^*=\left(\frac{2c_2}{15}+\frac{1}{3}, \frac{8c_2}{15}+\frac{1}{3}\right)
\end{equation}
which is also the unique equilibrium of the all-active game. The best responses of firm $1$ and firm $2$  for $c_1=\frac{5}{7}$ are shown on the left of Fig \ref{fig:br}.
\item If $c_2\in (\frac{5}{7}, \frac{3}{4})$, we obtain a continuum of Nash equilibria 
given by the set
$$
\mathcal X_1^*=\{x_1=[2x_2-1, x_2]\,,\,x_2 \in [5/7,c_2)\}
$$
The best responses of firm $1$ and firm $2$ in this case are shown on the center of Fig \ref{fig:br}. We remark that the Nash equilibria in $\mathcal X_1^*$ are not equivalent. In particular, the utility of firm $1$ is different in each Nash equilibrium, while the utility of firm $2$ is always $0$.
\item Finally, if $c_2 \in [\frac{3}{4},1]$, the set of Nash equilibria of the activation price  game is given by $\mathcal X^* =\tilde{\mathcal X}^*_1 \cup \mathcal X^*_2$ where
$$
\begin{aligned}
	\tilde{\mathcal X}_1^*&=\{x_1=[2x_2-1, x_2]\,,\,x_2 \in [5/7,3/4)\}\\
	\mathcal X_2^*&=\{x_1=\frac{1}{2}\,,\,x_2 \in [3/4,1]\}
\end{aligned}
$$
The best responses of firm $1$ and firm $2$  for $c_1=\frac{3}{4}$ are shown on the left of Fig \ref{fig:br}.	We remark that all Nash equilibria in the set $\mathcal X_2^*$ are equivalent from the point of view of utilities, while Nash equilibria in the set $\tilde{\mathcal X_1^*}$ give rise to different outcomes for firm $1$. 
\end{itemize}
Observe that the unique Nash equilibrium of the all-active game is given by $x^*$ in \eqref{eq:eq_2ac}
leading to the market-clearing price 
$\tilde p^{*}=\frac{2c_2+5}{9}$.
Observe that 
$$
x^{*}_2\leq \tilde p^{*}\quad \Leftrightarrow \quad c_2\leq \frac{5}{7}
$$
Therefore, for $c_2> \frac{5}{7}$, the Nash equilibria of the activation price game and the Nash equilibria of the all-active game do not coincide. On the other hand, if we consider the all-active game with $n=1$ agent with cost $c_1$, we obtain that the unique equilibrium is given by $x^*_1=\frac{1}{2}$ and $\tilde p^{*}=\frac{3}{4}$. For $c_2\geq \frac{3}{4}$ the unique equilibrium of the all-active game with one player gives the same outcome of the Nash equilibria in $\tilde{\mathcal X}^*_1$. Anyway, the activation price game admits also the equilibria in $\mathcal X^*_2$. Finally, for $c_2\in (\frac{5}{7}, \frac{3}{4})$, there is no correspondence from Nash equilibria of all-active games and Nash equilibria of the activation price game.
\end{example}

\section{The affine demand and quadratic costs case}\label{sec:ld_qc}
In this section, we study Nash equilibria of the $\lip$-PAB auction game in the special case when the demand function is affine as in \eqref{eq:aff_dem}  and the production cost functions are quadratic, specifically when  there exist constants $\cost>0$ and $c_i> 0$, for every producer $i$ in $\mc N$, such that 
\be\label{eq:quad_c} C_i(q) = \cost q+c_iq^2/2\,.\ee 
Observe that \eqref{eq:quad_c} implies that the marginal production costs of the producers 
$C_i'(q)=\cost +c_i q$ are affine and homogeneous in $0$. 
Hence, Corollary \ref{coro:uniqueness} implies that, for every $K>0$, the Nash equilibrium of the $\lip$-PAB auction game is unique up to the clearing price. In the following, we shall provide an explicit expression of such Nash equilibrium.  

{
Towards this goal, consider two vectors $d$ and $a$, respectively defined by, for each $i$ in $\mc N$,
\be\label{eq:quadr_coeff2}
d_{i} =\frac{K(n-1)+\gamma }{Kc_i(n-1)+c_i\gamma+n+\gamma/K}\,,
\ee

\be\label{eq:quadr_coeff}
a_i=\frac{K}{Kn+\gamma}-\frac{d_i}{Kn+\gamma}
\ee
Define, moreover,
\be\label{Delta-def}\Delta=\sum_{i\in\mc N}d_i\,.\ee

}


Then, we have the following result. 

\begin{theorem}\label{theo:linear}
Consider a set of producers $\mathcal N$, quadratic production cost functions as in \eqref{eq:quad_c} for every $i$ in $\mc N$, and affine demand function \eqref{eq:aff_dem}. 
Then, for every $K>0$, every Nash equilibrium $S^*$ of the $\lip$-PAB auction game satisfies \eqref{unique-eq}, with activation price {vector
\be\label{eq:ne_lin}
x^* = \ds\left(n+\frac{\gamma}{K}\right)\frac{\hat p\gamma+b\Delta}{\gamma+\Delta}a+\frac bKd\,,
\ee
}and clearing price  
\be\label{eq:ep_ne_lin}
p^*=\frac{\hat p\gamma+b\Delta}{\gamma+\Delta} \,.
\ee
\end{theorem}
\begin{proof}
{
The homogeneity of the marginal costs in $0$ allows to apply Theorem \ref{th:eq_cost}. This together with  Propositions \ref{prop:K=1} and \ref{prop:restr_game}(ii)  imply that every Nash equilibrium $S^*$ of the $\lip$-PAB auction game satisfies \eqref{unique-eq}, with activation price vector $x^*$ coinciding with the unique Nash equilibrium of the activation point game with set of producers $\mc N$, demand function $\ov D(p)=D(p)/K=\ov{\gamma}(\hat p-p)$, where $\ov\gamma=\gamma/K$, and production cost functions $\ov C_i(q)=C_i(K q)/K$ for every $i$ in $\mc N$. 

By substituting $\ov C_i'(q)=b+c_iKq$ and $\ov{\gamma}=\gamma/K$ into \eqref{eq:part_du}, we obtain
\be\label{new-eq}
\ba{rcl}0&=&\ds\tilde p -(n+\ov\gamma)x^*_i+(n+\ov\gamma-1)\ov C_i'(\tilde p -x^*_i) \\[7pt] 
&=&\ds\tilde p -(n+\gamma/K)x^*_i\\[7pt] &&
+(n+\gamma/K-1)(c_i K(\tilde p -x^*_i)+b)\,,\ea
\ee
for every $i$ in $\mc N$,  where 
\be\label{ptilde}\pa=\frac{\hat p\gamma/K+\sum_{j=1}^{n}x_j^*}{n+\gamma/K}\,.\ee
is the market clearing price.

%
%
%
%
%
%
Substituting \eqref{ptilde} into \eqref{new-eq}, multiplying by $(n+\gamma/K)$, and re-arranging terms, we get that the activation price vector satisfies the equation
\be\label{eq:foc_lin_ex}
x^*=a\1^Tx^*+\frac{\hat p\gamma}Ka+\frac bKd\,.
%
\ee
where $a$ and $d$ are defined in \eqref{eq:quadr_coeff} and \eqref{eq:quadr_coeff2}, respectively. Left multiplying by the vector $\1^T$ we now obtain
\be\label{average}\1^Tx^*=\frac{\frac{\hat p\gamma}K\1^Ta+\frac bK\1^Td}{1-\1^Ta}\ee
Substituting expression \eqref{average} into \eqref{eq:foc_lin_ex}, recalling definition \eqref{Delta-def}, and noticing that $1-\1^Ta=\frac{\Delta+\gamma}{Kn+\gamma}>0$, we obtain \eqref{eq:ne_lin}.
Substituting expression \eqref{average} into \eqref{ptilde}, we finally get expression \eqref{eq:ep_ne_lin}.

}
\end{proof}

We conclude this section with some examples.
\begin{example}\label{ex:1_lin}
Consider an affine demand $D(p)=10-p$ with $N=10$ and $\gamma=1$ and let $K=1$. Let us assume that there are $n=3$ firms participating in the auction and that their costs are $C_1(q)=\frac{1}{4}q^2$, $C_2(q)=q^2$ and $C_3(q)=\frac{3}{2}q^2$, that is, $c_1=\frac{1}{2}$, $c_2=2$ and $c_3=3$. Then, according to 
Theorem \ref{theo:linear}, 
satisfies \eqref{unique-eq} with activation price $x^*=[ 2.19,3.37,3.71]$, resulting in the market-clearing price $p^*(x^*)=4.82$. The aggregate demand and supply at Nash equilibrium $S^*$ are shown on the left in Figure \ref{fig:ex_lin}. We remark that 
all Nash equilibria of the $\lip$-PAB auction game are identical to $S^*$ up to the point $p^*$ and yield to the same market-clearing price and the same utilities for all the producers. 
\end{example}
\begin{example}\label{ex:2_lin}
Let us now consider the same setting of Example \ref{ex:1_lin} except for the production cost of firm $3$ that is now $C_3(q)=50q^2$, with  $c_3=100$. In this case, we obtain vector of activation prices $x=[2.45,3.77,5.34]$ and market-clearing price $p^*(x^*)=5.39$ (see the middle of Figure \ref{fig:ex_lin}). Notice that $x^*_i<p^*(x^*)$ for all $i$, that is all firms are selling a nonzero quantity, including firm $i=3$ that has very high costs. This follows from the assumption  that $C'_i(0)=0$ for all $i$. 
\end{example}
\begin{example}\label{ex:3_lin}
Let us consider the same demand in Example \ref{ex:3_lin} and let $c=[\frac{1}{2},2,3]$ and $K=100$. In this case, we obtain the activation prices  
$x^*_1=2.60$, $x^*_2=2.6395$ and $x^*_3=2.6439$, and the market-clearing price $p^*(x^*)= 2.65$. The aggregate supply at the Nash equilibrium $S^*$ of the $\lip$-PAB auction game is shown on the right of Figure \ref{fig:ex_lin}. Observe that, as $K$ grows large, the activation prices at Nash equilibrium are very close to the market-clearing price $p^*$. Also, the market-clearing price resulting when $K=100$ is lower than the market-clearing price resulting when $K=1$. These observations will be addressed in detail in the following section.
\end{example}
\begin{figure}
\centering
\includegraphics[width=0.32\textwidth]{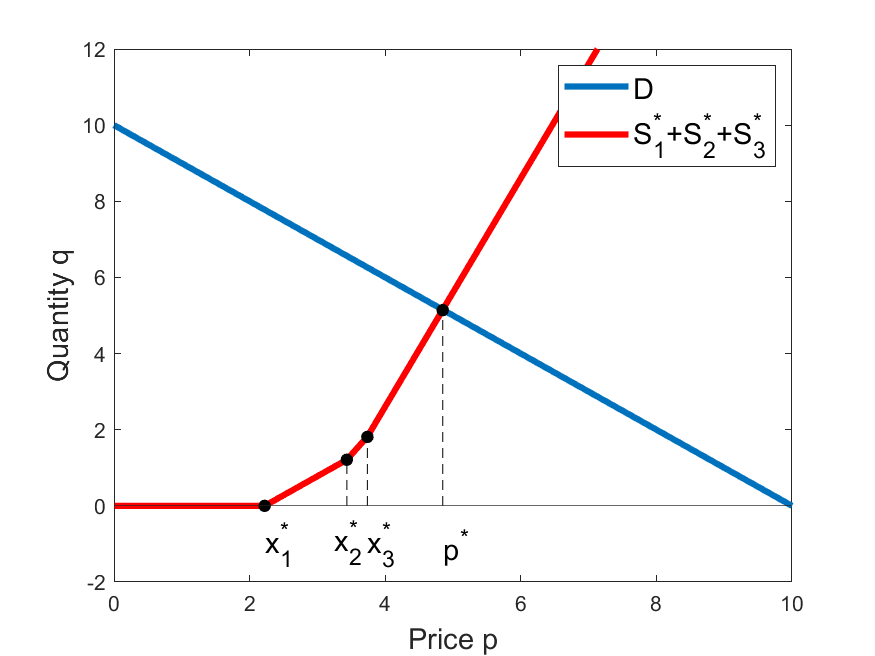}
\includegraphics[width=0.32\textwidth]{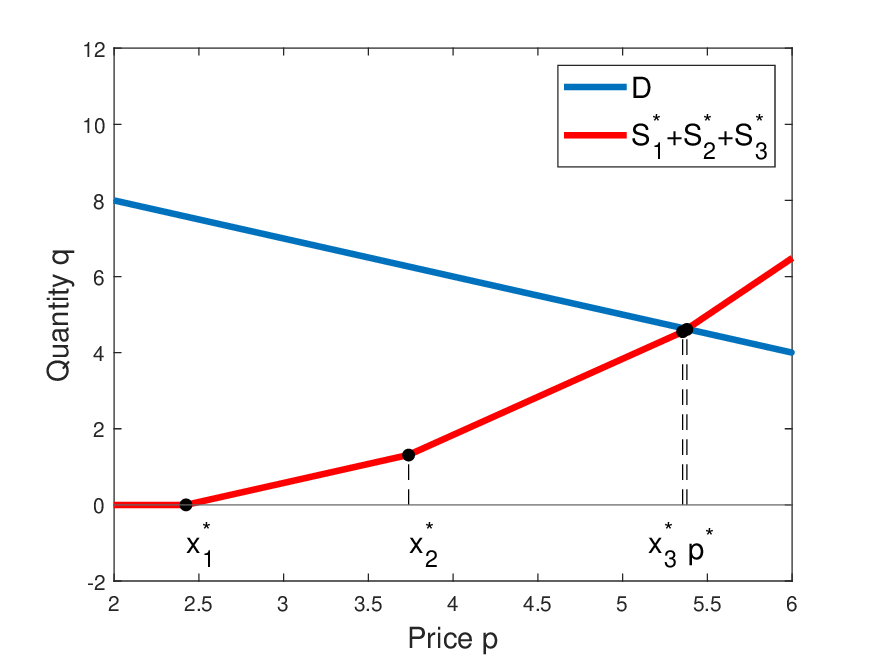}
\includegraphics[width=0.32\textwidth]{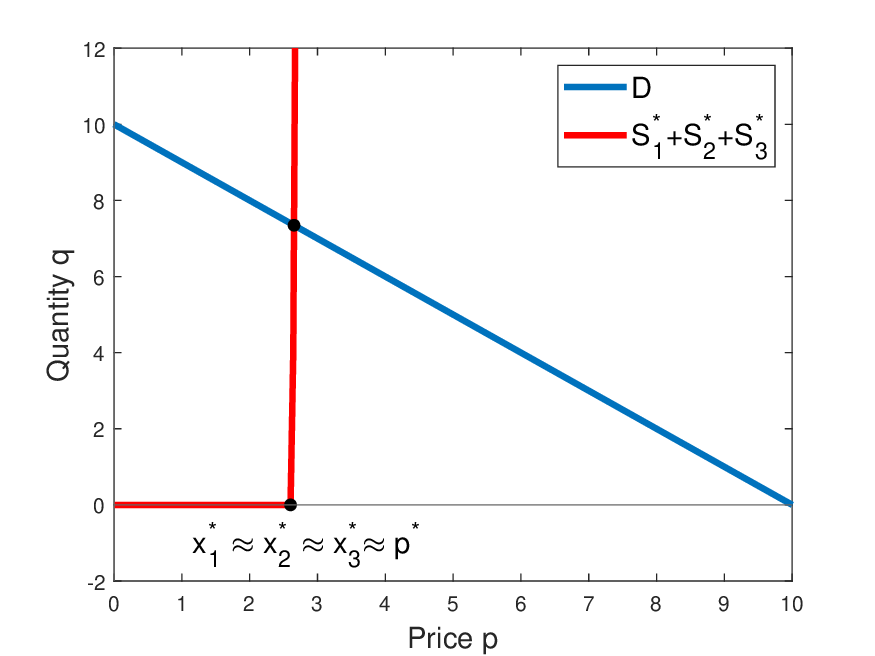}
\caption{Aggregate supply and demand at Nash equilibrium in the settings of Example \ref{ex:1_lin}, \ref{ex:2_lin} and \ref{ex:3_lin} (from left to right, respectively).}
\label{fig:ex_lin}
\end{figure}
Example \ref{ex:3_lin} leads to a discussion on the value of $K$. 
In the following result, we compute explicitly the resulting market-clearing price when $K$ goes to infinity. This limit permits a comparison with equilibrium outcome in standard oligopoly models, such as Cournot and Bertrand, and in the Supply Function equilibrium. 

\begin{proposition}\label{pr:limit}
Consider a set of producers $\mathcal N$, quadratic production cost functions as in \eqref{eq:quad_c} for every $i$ in $\mc N$, and affine demand function \eqref{eq:aff_dem}. 
For $K>0$, let $x_i^{(K)}$, for $i$ in $\mc N$ and $p^*_K$ be the activation prices and, respectively, the clearing price of any Nash equilibrium $S^{(K)}$ of the $\lip$-PAB auction game. Then, for every producer $i$ in $\mc N$, 
\be\label{eqlimxip}\lim_{K\to+\infty}x_i^{(K)}=\lim_{K\to+\infty}p^*_K=\frac{\hat p\gamma+b\sum_{j=1}^{n}1/c_j}{\gamma+\sum_{j=1}^{n}1/c_j}.\ee
Moreover, the asymptotic quantity of good sold by a producer $i$ in $\mc N$ is
\be\label{eqlimSK}\lim_{K\to+\infty}S_i^{(K)}(p^*_K)=\frac{\gamma/c_i}{\gamma+\sum_{j=1}^{n}1/c_j}[\hat p-b]_+\,,\ee
while its asymptotic profit is given 
\be\label{eqlimuiK}\lim_{K\to+\infty}u_i^{(K)}(p^*_K)=\frac{\gamma^2}{c_i(\gamma+\sum_{j=1}^{n}1/c_j)^2}[\hat p-b]^2_+\,.\ee
\end{proposition}
\begin{proof}
Let $a_i$, $d_i$, and $\Delta$ be defined as in \eqref{eq:quadr_coeff}, \eqref{eq:quadr_coeff2}, and \eqref{Delta-def}, respectively. Then, 
\be\label{limKai}
\lim_{K\to+\infty}a_i=
\frac1n\,,\;\; \lim_{K\to+\infty}	d_{i} =\frac1{c_i}\,,\;\; \lim_{K\to+\infty}\Delta=\sum_{i\in\mc N}\frac1{c_i}\,,
\ee
It then follows from \eqref{eq:ne_lin} and \eqref{limKDelta} that 
$$\lim_{K\to+\infty}x_i^{(K)} = \lim_{K\to+\infty}\left(n+\frac{\gamma}{K}\right)\frac{\hat p\gamma+b\Delta}{\gamma+\Delta}a_i+\frac bKd_i=\frac{\hat p\gamma+b\sum_{j=1}^{n}1/c_j}{\gamma+\sum_{j=1}^{n}1/c_j}\,,$$
for every producer $i$ in $\mc N$, while \eqref{eq:ep_ne_lin} and \eqref{limKDelta} imply that
$$\lim_{K\to+\infty}p^*_K=\lim_{K\to+\infty}\frac{\hat p\gamma+b\Delta}{\gamma+\Delta}=\frac{\hat p\gamma+b\sum_{j=1}^{n}1/c_j}{\gamma+\sum_{j=1}^{n}1/c_j} \,,$$
thus proving \eqref{eqlimxip}. Moreover, observe that from \eqref{eq:quadr_coeff} we get that
\be\label{limK}
\begin{aligned}
\lim_{K\to+\infty}&K\left(1-\left(n+\frac\gamma{K}\right)a_i\right)=\\=&\lim_{K\to+\infty}\frac{K(n-1)+\gamma}{Kc_i(n-1)+c_i\gamma+n+\gamma/K}=\frac1{c_i}\,,
\end{aligned}\ee
for every producer $i$ in $\mc N$, so that
$$\ba{rcl}\ds\lim_{K\to+\infty}S^{(K)}(p^*_K)
&=&\ds\lim_{K\to+\infty}K[p^*_K-x_i^{(K)}]_+\\
&=&\ds\lim_{K\to+\infty}\left[\frac{\hat p\gamma+b\Delta}{\gamma+\Delta}K\left(1-\left(n+\frac\gamma{K}\right)a_i\right)-bd_i\right]_+\\
&=&\ds\left[\frac{\hat p\gamma+b\sum_{j=1}^{n}1/c_j}{\gamma+\sum_{j=1}^{n}1/c_j}-b\right]_+\frac1{c_i}\\
&=&\ds\frac{\gamma/c_i}{\gamma+\sum_{j=1}^{n}1/c_j}[\hat p-b]_+
\ea$$
where the second equality follows from \eqref{eq:ep_ne_lin} and \eqref{eq:ne_lin} and the third one from \eqref{limK} and \eqref{limKdi}.  
Hence, we have proved \eqref{eqlimSK}. Finally, by \eqref{eq:utility_pab}, 
$$\ba{rcl}	\ds u_i^{(K)}(p^*_K)
&=&\ds p^*_KS_i^{(K)}(p^*_K)-\int_0^{p^*_K}S _i^{(K)}(p)\dd p - C^{(K)}_i(S_i^{(K)}(p^*_K))\\
&=&\ds p^*_KS_i^{(K)}(p^*_K)-K\int_{x_i^{(K)}}^{p^*_K}(p-x_i^{(K)})\dd p -C^{(K)}_i(S_i^{(K)}(p^*_K))\\
&=&\ds p^*_KS_i^{(K)}(p^*_K)-\frac {1}{2K}(S_i^{(K)}(p^*_K))^2 -\frac{c_i}{2} \left(S_i^{(K)}(p^*_K)\right)^2-bS_i^{(K)}(p^*_K)\,.
\ea$$
Therefore,
$$\ba{rcl}\ds\lim_{K\to+\infty}u_i^{(K)}(p^*_K)
&=&\ds\frac{\gamma^2(\hat{p}-b)}{c_i(\gamma+\sum_{j=1}^{n}1/c_j)^2}[\hat p-b]_+-\frac{\gamma^2}{2c_i(\gamma+\sum_{j=1}^{n}1/c_j)^2}[\hat p-b]^2_+
\\
&=&\ds\frac{\gamma^2}{2c_i(\gamma+\sum_{j=1}^{n}1/c_j)^2}[\hat p-b]^2_+
\,,\ea$$
where the fifth equality follows from \eqref{eqlimxip} and \eqref{eqlimSK}.
\end{proof}

We  conclude this section with the following fundamental remark, showing that, according to Proposition \ref{pr:limit}, in the limit, we obtain \emph{perfect competition} at the Nash equilibrium.
\begin{remark}
According to Proposition \ref{pr:limit}, for $K\rightarrow \infty$, we have that $\lim_{K\rightarrow \infty} x_i^{(K)}=\lim_{K\rightarrow \infty} p^*_K=p^*_{\infty}$ for all agents $i$ in $\mathcal N$. 
This observation implies that, as $K$ approaches infinity, 
the set of Nash equilibria of the $\lip$-PAB auction game converges to a set of \emph{step functions} that are zero up to the market-clearing price $p^*_{\infty}$.  This was also observed in Example \ref{ex:3_lin}. Therefore, by increasing the value of $K$, we are estimating the behavior of the PAB auction game with strategy space $\mathcal A=\cmon$ 
(recall the proof of Proposition \ref{pr:no_br} and Remark \ref{rem:step_f}).
Furthermore, in this setting, marginal costs are given by $C'_i(q_i) = c_i q_i+b$ and, for $\hat p\geq b$, we find
$$
\ba{rclclcl}
\ds\lim_{K\to+\infty}C_i'\left(S_i^{(K)}(p^*_K)\right)
&=&\ds C_i'\left(\frac{\gamma/c_i}{\gamma+\sum_{j=1}^{n}1/c_j}[\hat p-b]_+\right)\\[10pt]
&=&\ds C_i'\left(\frac1c_i\left(\frac{\hat p\gamma+b\sum_{j=1}^{n}1/c_j}{\gamma+\sum_{j=1}^{n}1/c_j}-b\right)\right)
&=&\ds C_i'\left(\frac{p^*_\infty-b}{c_i}\right)&=&\ds p^*_\infty
\,,\ea$$
for every $i$ in $\mathcal N$. 
Therefore, in the limit, we obtain \emph{perfect competition} (i.e., the price is equal to the marginal costs). As we shall see in the next section, when the costs are quadratic and homogeneous,  the unique equilibrium of the activation price game corresponds, in the limit, to the "best" Nash equilibrium of the Bertrand competition, that is, the one corresponding to perfect competition. Furthermore, we remark that we obtain  \emph{efficient allocation} also when the costs are heterogeneous.
\end{remark}

\subsection{Comparative statics}\label{ss:comparative}
We now aim to compare the market-clearing price at Nash equilibrium 
in the $\lip$-PAB auction as $K$ goes to infinity with those of other oligopoly models. 
Throughout, we consider a market with $n$ agents, an aggregate affine demand as in \eqref{eq:aff_dem} and quadratic costs as in \eqref{eq:quad_c}. 
\subsubsection{Cournot and Bertrand competition}	
In the following two examples, we compute the Nash equilibria of two famous oligopoly models \cite{microeconomics}, i.e., the Cournot competition \cite{cournot1838recherches}, which is a "quantity offering" game, and the Bertrand competition \cite{bertrand1883review}, which is a "price offering" game. 
In these examples and in the following comparison, we consider \textit{symmetric} quadratic costs as in \eqref{eq:quad_c} with $c_i =c>0$ for all $i$ in $\mathcal N$ and $b=0$. 

In the Cournot competition, 
firms compete in quantities, i.e., their strategies are the quantities $x_i \in \R^+$ that they will produce.  
The market-clearing price at which each quantity $x_i$ is remunerated is computed through the inverse demand function $P(q) = D^{-1}(q)$. 
In our setting, the Cournot competition admits the unique symmetric Nash equilibrium
$$
x^* = \frac{\gamma\hat p}{1+n+c\gamma}\,.
$$
If we substitute the Nash equilibrium $x^*$ in the inverse demand function, we obtain the market-clearing price 
\begin{equation}\label{eq:p_cournot}
p^*_\text{C}=\frac{1}{\gamma}\left(\gamma \hat p -\frac{\gamma\hat p}{1+n+c\gamma}\right)=\gamma \hat{p}\left(\gamma+\frac{n\gamma}{1+\gamma c}\right)^{-1}\,.
\end{equation}	

In the Bertrand model, 
firms set the price and, if they bid the lowest price, their output quantity is given by the market demand at that price, divided equally with all the other firms bidding that price. Otherwise, the awarded quantity is zero. 
The strategy of an agent $i$ in $\mathcal N$ is then the marginal price $x_i \in \R^+$. 
In this setting, we find a continuum of Nash equilibria, that is, for every 
$\alpha$ in $\left[0, \frac{n^2}{1+n}\right]$
\begin{equation}\label{eq:eq_ber}
x^*(\alpha) = \frac{\gamma\hat pc}{\gamma c +2(n-\alpha)} 
\end{equation}
defines a symmetric Nash equilibrium of the game. Since strategies in the Bertrand competition are prices, we also have that, for every $\alpha$ in $\left[0, \frac{n^2}{1+n}\right]$, the equilibrium price is given by
\begin{equation}\label{eq:eq_p_ber}
p^*_\text{B}(\alpha) = x^*(\alpha)= \gamma \hat p \left(\gamma + \frac{2(n-\alpha)}{c}\right)\,.
\end{equation}
See \cite{vanelli2024game} for further details on the computation.
We now wish to compare the market-clearing prices obtained in the previous examples with the market-clearing price of a Nash equilibrium of the $\lip$-PAB auction game, when $K$ goes to infinity, which in this setting, according to \eqref{eqlimxip}, is given by
$
p^*_\infty = \gamma \hat{p}\left(\gamma+{n}/c\right)\,.
$
Few algebraic computations lead to the following inequality: 
\begin{equation}\label{eq:comp1}
p^*_\text{B}(0)<p^*_{\infty}<	p^*_\text{C}\,,	\end{equation}
that holds for any $\hat p, \gamma>0$ and $c>0$.	Then, the market-clearing price at Nash equilibrium in the $\lip$-PAB auction as $K$ goes to infinity lies intermediate between the lowest Bertrand equilibrium and the unique market-clearing price at Nash equilibrium in the Cournot competition. 
Observe that, for $\alpha=\frac{n}{2}$ in $ \left[0, \frac{n^2}{1+n}\right]$, we obtain $p_\text{B}^*\left({n}/{2}\right)=p^*_\infty\,.$ Therefore, in this example, the unique Nash equilibrium of the activation price game corresponds to a Nash equilibrium of the Bertrand competition, and in particular it corresponds to the unique one achieving \emph{perfect competition}. Indeed, one can easily verify that $\alpha$ in $\left[0, {n^2}/(1+n)\right]$ satisfies
$$
C'\left(\frac{D(p^*_\text{B}(\alpha))}n\right)
=p^*_\text{B}(\alpha)
$$
for $C'(q)=cq$ and $D(p)=\gamma(\hat p-p)$ if and only if $\alpha={n}/{2}$. The strength of our result is that it allows to find perfect competition and efficient allocation also for \textit{heterogeneous costs}, which is challenging in Bertrand competitions. By imposing continuity in the supply functions, firms cannot win all the market demand by slightly reducing the price and this allows to achieve perfect competition at Nash equilibrium in several different settings.

\subsubsection{Supply Function Equilibria}
We now aim to compare the market-clearing price resulting in a $\lip$-PAB auction game as $K$ goes to infinity with market-clearing prices in Supply Function Equilibria (SFE) game models. 

The SFE game model was first proposed 
by Klemperer and Meyer \cite{sfe} in 1989. 
Unlike the Cournot competition, where firms choose the quantity to produce, and the Bertrand competition, where firms set the price, in the SFE game model firms' strategies are supply functions of price, i.e., $S_i \in \cmon$ as in \eqref{def-supply}, for all agents $i$ in $\mathcal N$. In the absence of uncertainty, the setting is analogous to the PAB auction game and the only difference between the two models lies in the total remuneration. 
SFE game models consider the uniform-price remuneration, leading to the utility, for all $i$ in $\mathcal N$,
$
u_i(S_i, S_{-i}) = p^*S_i(p^*)-C_i(S_i(p^*))\,,
$
where  $p^*=p^*(S)$ is the unique market-clearing price satisfying \eqref{eq:equilibrium}. The first fundamental observation made in \cite{sfe} is that, without uncertainty, there exists an infinite number of Nash equilibria. Such issue is solved when introducing \textit{uncertainty} in the model. 
Under uncertainty, the model and the assumptions are the same, except for the \textit{demand} that is now affected by an exogenous shock $\epsilon$,
which is a scalar random variable with density $f>0$ on $[\underline{\epsilon}, \bar{\epsilon}]$, with  $\bar{\epsilon}>\underline{\epsilon},\geq0$. The
demand function is then given by $D(p, \epsilon)$ with the further assumption that $D_p<0$, $D_{pp}\leq 0$, and $D_\epsilon>0$. 	In this setting, the authors look for ex-post Nash equilibria $(S_i, S_{-i})$, i.e., strategy profile that are Nash equilibria of the game for each realization of $\epsilon$ in $[\underline{\epsilon}, \bar{\epsilon}]$. 	
If we consider an affine demand function as in \eqref{eq:aff_dem}  with $\hat p\geq 0$ and $\gamma >0$ and quadratic costs as in \eqref{eq:quad_c} with  $c_i>0$, for $i$ in $\mathcal N$, we have that, according to \cite{linear_sfe} (eq. (6), pag. 10), linear supply functions of the form $S_i(p) = \beta_i p$
are SFE if and only if, 
for all $i$ in $\mathcal N$, $\beta_i\geq 0$ and 
\begin{equation}\label{eq:sfe_betas}
\beta_i = (1-c_i\beta_i)\left(\gamma +\sum_{j \neq i}\beta_j\right)\,.
\end{equation}
The resulting equilibrium price in this case is
\begin{equation}\label{eq:ep_sfe_qc_ad}
p^*_\text{SFE} = \frac{\gamma\hat p}{\gamma + \sum_{i=1}^n \beta_i }\,.
\end{equation}
Recall that, in the $\lip$-PAB auction game for $K\rightarrow \infty$, the market-clearing price is given by 
$$p^*_{\infty} =  \frac{\gamma\hat p}{\gamma + \sum_{i=1}^n 1/c_i }\,.$$	
Observe that, by definition of $\beta_i$ in \eqref{eq:sfe_betas}, $\beta_i\geq 0$ if and only if $1-c_i\beta_i \geq0$, thus implying $\beta_i\leq 1/c_i$. Furthermore, for $\beta_i= 1/c_i$, we obtain $\beta_i=0$, which is an absurd. Therefore, $\beta_i<\frac{1}{c_i}$, 
leading to
$$p^*_{\infty}<p^*_\text{SFE}\,.$$
Thus, the market-clearing price at Nash equilibrium in the $\lip$-PAB auction as $K$ goes to infinity is always strictly lower than the market-clearing price in the SFE. Anyway, we remark that, differently from SFE, our model does include uncertainty in the demand.

\begin{figure}
\centering
\includegraphics[width=0.31\textwidth]{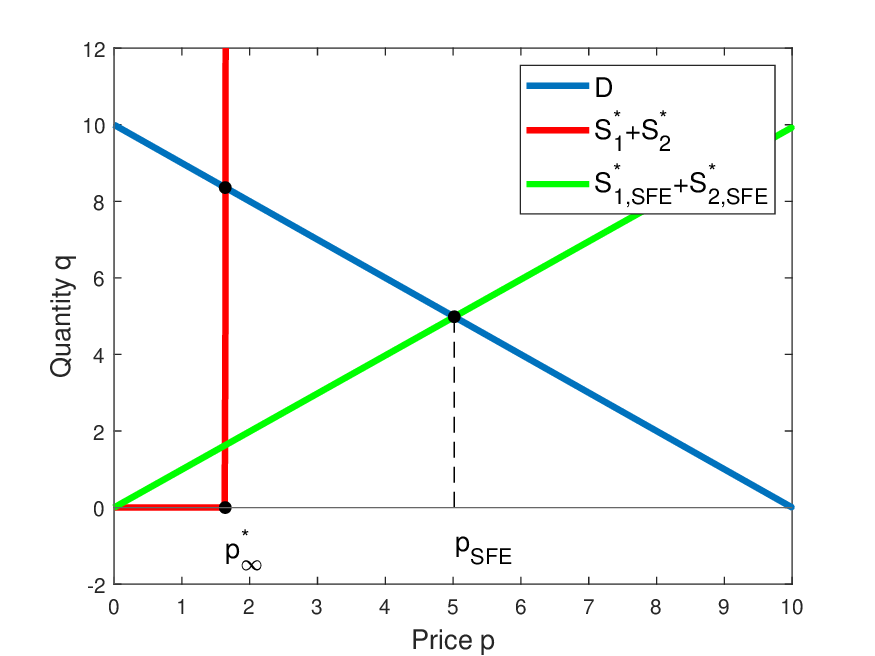}
\includegraphics[width=0.31\textwidth]{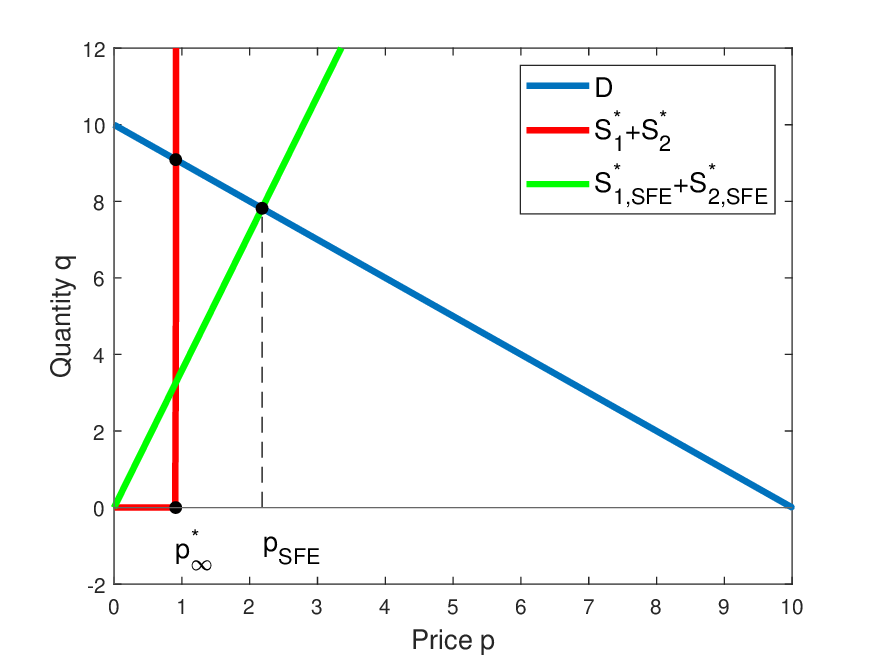}
\includegraphics[width=0.31\textwidth]{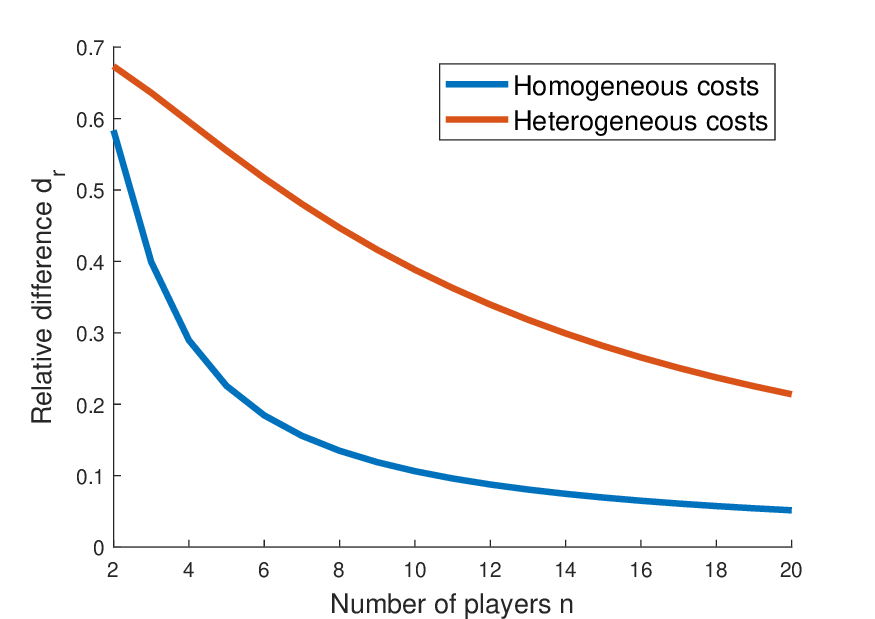}
\caption{On the left, comparison between the Nash equilibrium of the $\lip$-PAB auction game and Supply Function equilibria (see Example \ref{ex:limit1}). On the right, relative difference between $p^*_{SFE}$ and $p^*_\infty$ for different number of players and distribution of costs (see Example \ref{ex:limit2}).}
\label{fig:limit1}
\end{figure}
\begin{example}\label{ex:limit1}
Let us consider $n=2$ firms with production costs $C_1(q)=0.1q^2$, $C_2(q)=5q^2$ and let $D(p)=10-p$. For $K$ large, for instance, $K=1000$, we obtain a market-clearing price approaching $p^*_\infty = 1.8$. If we compute the Supply Function equilibrium in the same setting, we obtain $p_{SFE}=5.1$. In this case (see Figure \ref{fig:limit1}), we can observe a big difference between $p^*_\infty$ and $p_{SFE}$. This is due to the fact that there are only two players and and the costs are very heterogeneous. If we consider $C_1(q)=0.1q^2$,$C_2(q)=0.1q^2$, we still observe a significant difference between the two prices $p^*_\infty = 1 $ and $p_{SFE}=2.2$ (see the right of Figure \ref{fig:limit1}). In general, the differences between the two mechanisms are more evident when there is heterogeneity in the costs and/or the costs are very low. 
\end{example}
\begin{example}\label{ex:limit2}
Let $D(p)=10-p$ and $K=1000$. If we consider homogeneous agents with $c_i = c=0.02$ for all $i$ in $\mathcal N$, and we compute the relative difference  $d_r = \frac{p^*_{SFE}-p^*_\infty}{p^*_{SFE}}$ we can observe that it decreases with the number of agents $n$. Anyway, it is still significant. In particular, if we consider heterogeneous agents with $c_i$ uniformly distributed in $[0.02, 5]$ we obtain that $d_r = \frac{p^*_{SFE}-p^*_\infty}{p^*_{SFE}}$ is larger and even more significant compared to the homogeneous case (see Fig \ref{fig:limit1})
\end{example}
\section{Conclusion}\label{sec:conclusions}
In this paper, we introduce and analyze the pay-as-bid auction game, a model incorporating supply function strategies, pay-as-bid remuneration, and asymmetric firms. 
Our fundamental finding is that, by limiting the strategy space to Lipschitz-continuous supply functions, pure-strategy Nash equilibria exist and they can be represented as piece-wise affine functions with slope equal to the maximum allowed Lipschitz constant.  More precisely, Nash equilibria of the $\lip$-PAB auction game can be fully characterized starting from Nash equilibria of a parameterized game  with continuous scalar actions: the  activation price game. This result paves the way to our comprehensive analysis of the game. 
When the demand is affine and the costs are homogeneous in zero, we demonstrate uniqueness of Nash equilibria up to the market-clearing price. We further compute a closed form expression for the case when the costs are quadratic and  we show that the market-clearing price of the pay-as-bid auction game as $K$ goes to infinity achieves perfect competition. 


Ongoing work involves validating the model with data from the Italian electricity market. To better capture market complexities, future research directions include addressing heterogeneous costs in zero, capacity constraints, network structure, and demand uncertainty. Finally, we would like to investigate the combination of a uniform-price auction and a pay-as-bid auction as a two-stage game. Drawing inspiration from the structure of existing electricity markets, understanding the interplay between these auction mechanisms can offer valuable insights into their combined impact and potential benefits. 

%
%
%


\appendix
\section{Proof of Proposition \ref{prop:quasi-concavity}}\label{sec:proof-prop:quasi-concavity}
During the proof we will often use the following jargon. We say that a function $\alpha: [a,b]\to\R$ is of class $\mathcal C^k[a,b]$ if its derivatives exist up to the $k$-th order and are continuous, including the boundary points where we consider, respectively,  right derivatives in $a$ and left derivatives in $b$. We indicate the derivatives also at boundaries with the usual notation $\alpha'(z)$, $\alpha''(z)$, etc.

Let us now fix $i$ and $x_{-i}$. We define $\varphi:[0, \hat p]\to[0, \hat p]$ such that for every $z$ in $[0, \hat p]$, $\varphi(z)=p^*(z, x_{-i})$ is the unique market-clearing price when $x_i=z$. Furthermore, we denote  $u(z):= \ov u_i(z, x_{-i})$ and, from expression (\ref{eq:u_r}), using some algebra and Proposition \ref{prop:increasing} (iii), we obtain the representation 
\begin{equation}\label{eq:u}
\small	u(z) = \begin{cases}
\ds\frac{1}{2}\varphi^2(z)-\frac{1}{2}z^2-C_i(\varphi(z)-z) \; &0\leq z\leq \hat p_i(x_{-i})\,,\\[5pt]
-C_i(0) &\hat p_i(x_{-i})<z\leq\hat{p}\,.
\end{cases}
\end{equation}
We now study the behavior of $\varphi$ and then of $u$ on $[0, \hat p_i(x_{-i})]$, that is, for the values $z$ satisfying $z\leq\varphi(z)$. We notice that on $[0, \hat p_i(x_{-i})]$, equation \eqref{eq:equilibrium_r} determining the function $\varphi$ can be written as
$z=\varphi(z)+\sum\limits_{j\neq i}[\varphi(z)-x_{j}]_+-D(\varphi(z))$. This implies that the function
$f:[0, \hat p]\to\R$ given by
\begin{equation}\label{eq:f}
f(w)=w+\sum\limits_{j\neq i}[w-x_{j}]_+-D(w)
\end{equation}
satisfies
$z=f(\varphi(z))$  for every $z\leq \hat p_i(x_{-i})$.
As $f$ is strictly increasing and thus invertible, we have that $\varphi(z)=f^{-1}(z)$ for $z$ in $[0,\hat p_i(x_{-i})]$. Notice that $f$ is continuous and piecewise $\mathcal C^2$ with lack of derivatives at points $\{x_j\;\, j\neq i\}$ and moreover that $f'(z)>1$ wherever the derivative exists (also as right and left derivative at points $x_j$'s).
Basic calculus shows that $\varphi(z)$ is also continuous on $[0, \hat p_i(x_{-i}]$ and $\mathcal C^2$ except at points in 
$$\mathcal D=\{f(x_j)
\;|\, j\neq i,\, x_j\in [0, \hat p_i(x_{-i}]\}\,.$$	
We label points in $\mathcal D\cup\{0, \hat p_i(x_{-i}\}$  as $z_1< z_2<\cdots< z_q$  and we consider the intervals $I_k=[z_k, z_{k+1}]$ for $k=1,\dots , q-1$. We fix one such interval $I_k$.  Then, the restriction $\phi_{|I_k}$ is of class $\mathcal C^ 2(I_k)$ and the following estimations hold true for $z\in I_k$:
\begin{equation}\label{eq:dphi}
\begin{aligned}
0 < \varphi'(z) &= \frac{1}{f'(\varphi(z))}<1
\,;
\end{aligned}
\end{equation}
\begin{equation}\label{eq:dphi2}
\varphi''(z) = \frac{D''(\varphi(z)) \varphi'(z)}{(f'(\varphi(z)))^2}\leq 0\,,
\end{equation}
because of the properties of $f$ and the standing assumption on $D$. 
Similarly, $u_{|I_k}$ is of class $\mathcal C^2(I_k)$ and we have for $z\in I_k$,
\begin{equation}\label{u-derivative}
\begin{aligned}
u'(z)=&\varphi(z)\varphi'(z)-z-C_i'(\varphi(z)-z)(\varphi'(z)-1)  \\ 	
=&(\varphi(z)-C_i'(\varphi(z)-z))(\varphi'(z)-1)\\&+\varphi(z)-z\,.
\end{aligned}
\end{equation}
We now study the sign of $u'$ on $I_k$ showing, in particular, it can only change sign once from positive to negative.
If $z^*\in I_k$ is such that $u'(z^*)=0$, it follows from (\ref{eq:dphi}), (\ref{u-derivative}) and  the fact that $\varphi(z)\geq z$ for every $z\in[0,\hat p_i(x_{-i})]$, that 
$$\varphi(z^*)-C_i'(\varphi(z^*)-z^*) =\underbrace{(\varphi(z^*)-z^*)}_{\geq0}/\underbrace{(1-\varphi'(z^*))}_{>0}\geq 0\,.$$
%
%
and, therefore, by \eqref{eq:dphi}, \eqref{eq:dphi2} and the standing assumptions on $C_i$, 
\begin{equation}\label{eq:conc_sp}
\begin{aligned}
u''(z^*)\small{=} &\underbrace{\varphi'(z^*)^2-1}_{<0}-\underbrace{C_i''(\varphi(z^*)-z^*)}_{\geq 0}(\varphi'(z^*)-1)^2\\&+\underbrace{(\varphi(z^*)-C_i'(\varphi(z^*)-z^*))}_{\geq 0}\underbrace{\varphi''(z^*)}_{\leq 0}< 0 \,.
\end{aligned}
\end{equation}
Therefore, the sign of $u'$ is locally strictly positive at the left of $z^*$ and strictly negative at the right of $z^*$. This implies that at most one stationary point can be present. 
In case a stationary point $z^*$ exists, there are only three possible things that can happen. The first possibility is that $z^*\in \mathring{I_k}$: in this case $u'(z)>0$ for $z\in [z_k, z^*[ $ and $u'(z)<0$ for $z\in ]z^*, z_{k+1}]$. The other two possibilities are that $z^*$ is on the boundary: either $z^*=z_k$ and $u'(z)<0$ for all $z\in ]z_k, z_{k+1}]$ or $z^*=z_{k+1}$ and $u'(z)>0$ for all $z\in [z_k, z_{k+1}[$. When one of such situations happens, we refer to the interval $I_k$ as, respectively, to a $(+\,0\,-)$ interval, a $(0\,-)$ interval, and a $(+\,0)$ interval. We refer to $I_k$ as to a $(-)$ interval or to a $(+)$ interval if instead $u'(z)<0$, respectively, $u'(z)>0$ for every $z\in I_k$.
Finally, it follows from \eqref{eq:dphi} and the fact that $f'_-(x_k)<f'_+(x_k)$ for every $k=2,\dots , q-1$ that
$$0\leq (\varphi_{|I_{k+1}})'(z_k)\leq(\varphi_{|{I_{k}}})'(z_k)<1\,.$$
This relation combined with \eqref{eq:dphi}, the fact that $\varphi(z)\geq z$ for every $z\in[0,\hat p_i(x_{-i})]$, and 
expression (\ref{u-derivative}) yield
$$\varphi(z_k)-C_i'(\varphi(z_k)-z_k)\geq 0 
\quad\Rightarrow \quad (u_{|I_k})'(z_k)\geq (u_{|{I_{k+1}}})'(z_k)\,,$$
$$\varphi(z_k)-C_i'(\varphi(z_k)-z_k)< 0 \quad \\[5pt]\quad\Rightarrow \quad 0 {<}(u_{|I_k})'(z_k)\leq (u_{|{I_{k+1}}})'(z_k)\,.$$

This implies that if $I_k$ is either a $(+\,0\,-)$ interval, or a $(0\,-)$ interval, or a $(-)$ interval, then $I_h$ is a  $(-)$ interval for every $h>k$. Similarly, 
if $I_k$ is either a $(+\,0\,-)$ interval, or a $(+\, 0)$ interval, or a $(+)$ interval, then $I_h$ is a  $(+)$ interval for every $h<k$. This implies the following. Either $u$ is monotonic: strictly decreasing (all $-$ intervals or all $-$ intervals but the first one that is a $(0\,-)$ interval), or strictly increasing (all $+$ intervals or all $+$ intervals but the last one that is a $(+\, 0)$ interval). Otherwise, a change of monotonicity can happen in two ways. Either inside a $(+\,0\,-)$ interval (necessarily with only $(+)$ intervals at its left and $(-)$ intervals at its right) and in this case $u$ is unimodal, first strictly increasing and after strictly decreasing. Otherwise, it can happen with a $(+)$ or a $(+\,0)$ interval followed by a $(-)$ or by a $(0\, -)$ interval. Arguing as in the previous case, we conclude that $u$ is in this case unimodal, first strictly increasing and after strictly decreasing. 

We are now ready to prove the three statements:

(i) Since the function $u$ is continuous on $[0, \hat p]$ and in the interval $[\hat p_i(x_{-i}), \hat p]$ is constant, also in the larger domain $[0, \hat p]$ maintains the above property: it is either monotonic, non-decreasing or non-increasing, or bimodal, first increasing and then non-increasing in $[0, \hat{p}]$.  For a classical result, this implies that $u$ is quasi-concave.

(ii) It is an immediate consequence of previous analysis of the modality of $u$ on $[0, \hat p_i(x_{-i})]$.

(iii) Consider the last interval $I_q$ in the analysis above and take $z^*=\hat p_i(x_{-i})$. This is a  maximum point if and only if $I_q$ is either a $(+)$ or a $(+\,0)$ interval and this is equivalent to asking that $u'(z)\geq 0$.
Since $\phi(z^*)=z^*$, we deduce that the sign of $u'(z^*)$ is the same of the sign of the quantity 
$\varphi(z^*)-C_i'(\varphi(z^*)-z^*)=\hat p_i(x_{-i})-C_i'(0)$.
This yields the result and completes the proof.
$\hfill\square$ 



\bibliographystyle{apalike}
\bibliography{bib} 


\end{document}